\documentclass[journal]{IEEEtran}
\usepackage{graphics} 
\usepackage{epsfig} 
\usepackage{graphicx}
\usepackage{mathptmx} 
\usepackage{times} 
\usepackage{amsmath} 
\usepackage{amssymb}  
\usepackage{amsbsy,bm,fixmath}
\usepackage{cases}
\newtheorem{theorem}{Theorem}
\newtheorem{definition}[theorem]{Definition}
\newtheorem{lemma}[theorem]{Lemma}
\newtheorem{proposition}[theorem]{Proposition}

\newtheorem{example}{Example}
\newtheorem{remark}{Remark}

\def\mbA{\mathbold{A}}
\def\mbB{\mathbold{B}}
\def\mbD{\mathbold{D}}
\def\mbK{\mathbold{K}}
\def\mbP{\mathbold{P}}
\def\mbQ{\mathbold{Q}}
\def\mbS{\mathbold{S}}
\def\mbr{\mathbold{r}}
\def\mbx{\mathbold{x}}

\def\endproof{\hfill$\Box$}

\ifCLASSINFOpdf
\else
\fi
\begin{document}

%
\title{Linear Quadratic Mean Field Games:
Asymptotic Solvability and Relation to the Fixed Point Approach\thanks{This work was  supported in part by Natural Sciences and Engineering
Research Council (NSERC) of Canada under a Discovery Grant and a Discovery Accelerator Supplements Program.}}
%
%
%

\author{Minyi~Huang
        and~Mengjie~Zhou
\thanks{M. Huang and M. Zhou are with the School of
Mathematics and Statistics, Carleton University, Ottawa, ON K1S 5B6,
Canada (mhuang@math.carleton.ca, mengjiezhou@cmail.carleton.ca).}
\thanks{IEEE Trans. Autom. Control, submitted Apr. 2018; revised Dec. 2018;
accepted May 2019.   }
}

\maketitle

\begin{abstract}
Mean field game theory has been developed largely  following two
routes. One of them, called the direct approach,  starts by solving a large-scale game and
next derives a set of limiting equations as the population size
tends to infinity. The second route is to apply mean field approximations
and formalize a fixed point problem by analyzing the best response of
a representative player.
This paper addresses the connection and difference of the two approaches
in a linear quadratic (LQ) setting.
 We first introduce an asymptotic solvability notion for the direct approach,
which means for all sufficiently large population sizes, the corresponding
game has a set of feedback Nash strategies in addition to a mild regularity requirement.
We provide a necessary and sufficient condition for asymptotic solvability
and show that in this case the solution
converges to a mean field limit. This is accomplished by developing a
re-scaling method to derive a low dimensional ordinary differential equation
(ODE) system, where a non-symmetric Riccati ODE has a central role.
We next compare with the
fixed point approach which determines a
two point boundary value (TPBV) problem,
and  show that asymptotic solvability implies feasibility of the
fixed point approach, but the converse is not true.
We further address non-uniqueness in the
fixed point approach and examine the long time behavior of the non-symmetric Riccati ODE in the asymptotic solvability problem.
\end{abstract}


%
\IEEEpeerreviewmaketitle

\section{Introduction}
%
%
%
%

\IEEEPARstart{M}{ean} field game (MFG) theory has undergone a phenomenal growth. It provides a powerful methodology
for tackling complexity in large-population noncooperative decision problems.
The readers are referred to \cite{BFY13,CHM17,C12,CD18,GS14} for an overview of the theory and applications.
The past developments have largely
followed two routes \cite{HCM07,HMC06,LL07} which are called,  respectively, the bottom-up and top-down
 approaches in \cite{CHM17}.

 One route starts by formally solving an $N$-player game to obtain a large coupled solution equation system. The next step is to derive a limit for the solution by taking $N\to \infty$ \cite{LL07},
 which can be called the direct
(or bottom up) approach; see route one in Fig. \ref{fig:dia}.
Another route is to solve an optimal control problem of a single agent
based on consistent mean field approximations and formalize a fixed point problem to determine the
mean field, and this is called the fixed point (or top-down) approach
  \cite{HCM07,HMC06} and also called Nash
certainty equivalence in  \cite{HMC06}; see route
two in Fig. \ref{fig:dia}. The solution of the fixed point problem may be used to
design decentralized strategies in the original large but finite population model to achieve an $\epsilon$-Nash equilibrium \cite{HCM07}.  Under such a set of strategies, each player can further improve little even if it can access centralized information of all players.
 Compared with Nash strategies determined under centralized information, the above solution has much lower complexity in its computation and implementation.

 The reader may consult further literature on the direct approach \cite{C12} and the fixed point approach \cite{BFY13,BSYY16,CD18,KTY14,LZ08}. Also see \cite{F14,L16} for the direct approach in a probabilistic framework.
 We note that the diagram in Fig. \ref{fig:dia} displays the basic theoretic framework of mean field games with all players being comparably small, called peers. When the model involves a major player or common
noise, the analysis has been extended for  the direct approach \cite{CDLL15} and the fixed point approach
 \cite{BFY13,CK17,CD18,CZ16,H10,NC13}.

So far the investigation of the connection and difference
between the two approaches regarding their scope of applicability
is scarce. Their systematic comparison  is generally
difficult since in the literature very often the analysis in each approach
is carried out under various sufficient
conditions.
 In this work we contribute in this direction
within the framework of linear-quadratic (LQ) mean field games
with a finite time horizon. The analysis of mean field games in the LQ setting has attracted substantial interest due to their appealing analytical structure
\cite{BSYY16,CK17,HWW16,HCM07,LZ08,MB15,NCMH13,SML18,TZB14,WZ13}. Specifically, the decentralized strategy of an individual player may be determined in a linear feedback form. Partial state information is  considered in \cite{CK17,HWW16}, and  \cite{HWW16} adopts linear backward stochastic differential equations
to model state dynamics.

In  this paper  we first study an asymptotic solvability problem initially introduced in \cite{HZ18},
 which may be viewed as an instance of the direct (i.e., bottom-up) approach.
We adopt an appropriately defined asymptotic
solvability notion for the sequence of LQ games with increasing population sizes so that a neat necessary and sufficient condition can be derived.
This will on one hand further our understanding of
the direct approach and on the other offer a foundation
for a thorough comparison with the fixed point approach.
We start with an entirely
conventional solution of the game by dynamic programming, which leads to  a set of coupled Riccati ODEs.
It turns out that the necessary and sufficient condition for asymptotic solvability
is characterized by a low dimensional non-symmetric Riccati ODE  derived by a novel re-scaling
technique. The methodology of identifying low dimensional dynamics
to capture essential information on high dimensional dynamical
behavior shares similarity to the statistical physics literature on mean field
oscillator models \cite{MBS09,OA08,PM14}.
This approach is also closely related to  an early problem of mean field social optimization, which studies a high dimensional algebraic Riccati equation (ARE) and uses
symmetry for dimension reduction \cite[Sec. 6.3]{H03}.
Other related works include \cite{H15,P14,P15}.
An optimal control problem for a set of symmetric agents with mean field coupling
is solved in \cite{H15} by a large-scale Riccati ODE, and a mean field limit is derived.
An LQ Nash game of infinite time horizon is analyzed in \cite{P14} where the number of players increases to infinity. By postulating the strategies of all players and  examining the control problem of a fixed player,  a family of low dimensional control problems and their parametrized AREs are solved by applying an implicit function theorem for which sufficient conditions are obtained for large population sizes. The solvability of LQ games with increasing population sizes in the set-up of \cite{LL07} is studied in \cite{P15} analyzing $2N$-coupled steady-state Hamilton-Jacobi-Bellman (HJB) and Fokker-Planck-Kolmogorov (FPK) equations under some algebraic conditions, where each player's control is restricted to be local state feedback from the beginning.

Subsequently the paper  investigates the
relation of the two fundamental approaches \cite{HCM07,HMC06,LL07} shown in Fig. \ref{fig:dia},
which has been made possible by the solution of the asymptotic solvability problem. In so doing, we first revisit  the fixed point approach
for the mean field game, and determine the necessary and sufficient condition
for the solvability of the resulting two point boundary value (TPBV) problem.
It is shown that asymptotic solvability provides a sufficient condition
for the TPBV problem  to be
solvable and in fact uniquely solvable in this case; this is due to the fact that one can use a non-symmetric Riccati ODE to decouple and solve a general linear TPBV problem \cite{F02}. However, there exist scenarios
for our TPBV problem to be solvable but asymptotic solvability fails.
This suggests non-equivalence of the two approaches in general.
We make  a further connection with the original work \cite{HCM07}, which applies the fixed point approach under a contraction condition; we show in this case asymptotic solvability  holds for the sequence of games.

Our study of the asymptotic solvability problem and
the subsequent comparison of the two fundamental approaches  provides new insights into the relation between the infinite population mean field game and large finite population games. Historically, the study of the relation between
large finite population games and their infinite population limit has been
a subject of great interest and importance \cite{A64,CP10,G84,HM85,MC83} although this is usually for static games.

\begin{figure}[t]
\begin{center}
\begin{tabular}{c}
\psfig{file=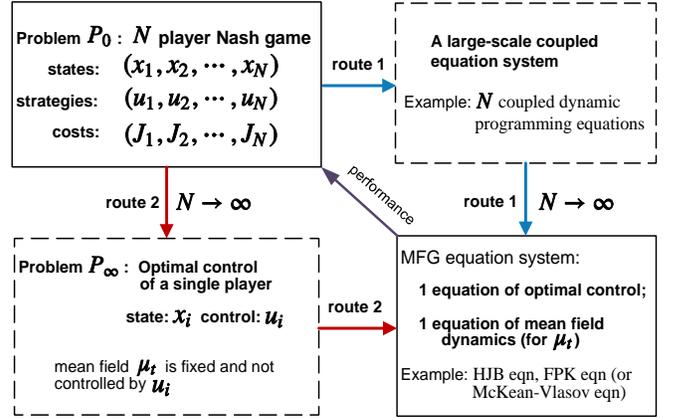, width=3.4in, height=2.2in}
\end{tabular}
\end{center}
\caption{The two fundamental approaches: The direct (or bottom-up) approach (see route 1) and  the fixed point (or top-down) approach (see route 2)}
 \label{fig:dia}
\end{figure}

For the TPBV problem in the fixed point approach we
further examine the non-uniqueness issue, which has been of
interest in the MFG literature;
see non-uniqueness results for nonlinear MFG models \cite{BF17,CT17,GNP17} and for an LQ example  with a non-quadratic terminal cost \cite{T18}.
Non-uniqueness has been well studied in the traditional literature of
LQ dynamic games; see \cite{E82,E00}.
 Finally, we analyze the long time behavior of
the non-symmetric Riccati ODE in the asymptotic solvability problem. The analysis
is related to a non-symmetric algebraic Riccati equation (NARE) and faces the issue of solution selection.
We introduce the notion of a stabilizing solution for the NARE and
derive the necessary and sufficient condition for its existence and uniqueness.

The main contributions of the paper are outlined as follows:
\begin{enumerate}
\item We study an $N$-player LQ Nash game and introduce the notion of asymptotic solvability, which  can be regarded as a direct approach in mean field games.

\item By a re-scaling technique,  a necessary and sufficient condition for asymptotic solvability is obtained in terms of a non-symmetric Riccati ODE.
 This lays down a foundation to address the exact relation of two fundamental approaches in mean field games: the direct approach and the fixed point approach. We show asymptotic solvability implies unique solvability of the TPBV problem in
the fixed point approach. We further show that a contraction condition of the fixed point approach introduced in the
original work \cite{HCM07} implies asymptotic solvability. We further determine conditions for non-uniqueness to occur in the fixed point approach.

\item The long time behavior of the non-symmetric Riccati ODE in the direct approach is studied. A necessary and sufficient algebraic condition is obtained for it to have a stabilizing solution.

\end{enumerate}

We make some convention on notation.
 Throughout the paper, $E$ is reserved for denoting the mean of a random variable or a random vector. For symmetric matrix $S\ge 0$, we may write $x^TSx=|x|_S^2$. We denote by ${\bf 1}_{k\times l}$  a $k\times l$ matrix with all entries equal to 1, by $\otimes$ the Kronecker product, and by the column  vectors  $\{e_1^k, \ldots, e_k^k\}$  the canonical basis of $\mathbb {R}^k$.  We may use a subscript $n$ to indicate the identity matrix $I_n$ to be $n\times n$. For a vector or matrix $Z$, $|Z|$  stands for its Euclidean norm.
 For an $l\times m$ real matrix $Z =(z_{ij})_{1\le i\le l,1\le j\le m}$, denote the $l_1$-norm $\|Z\|_{l_1}= \sum_{i,j}|z_{ij}|$.

The organization of the paper is as follows.
Section \ref{sec:mod} describes the LQ Nash game together with its solution via dynamic programming and Riccati ODEs.
Section \ref{sec:asy} presents the necessary and sufficient condition for asymptotic  solvability and derives decentralized strategies. We revisit the fixed point approach  in Section
\ref{sec:fp} and examine its relation to asymptotic solvability. To further study the relation of the two approaches,  Section \ref{sec:scal} develops in-depth analysis of the scalar individual state case. The long time behavior of the non-symmetric Riccati ODE is examined in Section \ref{sec:long}.
 Illustrative  examples are provided in Section \ref{sec:nume}.
Section \ref{sec:con} concludes the paper.

\section{The LQ Nash Game} \label{sec:mod}

Consider a population of $N$ players (or agents) denoted by ${\cal A}_i$,  $1\le i\le N$.
The state process $X_i(t)$ of ${\cal A}_i$ satisfies the following
stochastic differential equation (SDE)
\begin{align}
dX_i(t)=\big(AX_i(t)+Bu_i(t)+GX^{(N)}(t)\big)dt+DdW_i(t)&, \label{stateXi}\\
  1\le i\le N,  & \nonumber
\end{align}
where  we have state  $X_i\in \mathbb{R}^n$, control
$u_i\in\mathbb{R}^{n_1}$, and the coupling term
 $ X^{(N)}=\frac{1}{N}\sum_{k=1}^N X_k$. The constant matrices $A$, $B$, $G$, $D$ have compatible dimensions.
  The initial states $\{X_i(0), 1\le i\le N\}$
  are independent with $EX_i(0)=x_i(0)$ and finite second moment.
 The $N$ standard  $n_2$-dimensional Brownian motions $\{W_i, 1\le i\le N\}$ are independent and also independent of the  initial states.
The cost of player ${\cal A}_i$ in the Nash game is given by
\begin{align}\label{costJi}
J_i =\  & E\int_0^T \Big(  |X_i(t)-\Gamma X^{(N)}(t)-\eta|_Q^2+u_i^T(t) R u_i(t)\Big)dt \nonumber\\
&+E|X_i(T)-\Gamma_f X^{(N)}(T)-\eta_f|_{Q_f}^2.
\end{align}
The constant matrices (or vectors)
 $\Gamma$, $Q$, $R$, $\Gamma_f$, $Q_f$, $\eta$, $\eta_f$ above
have compatible dimensions, and we have $Q\ge 0$, $R>0$, $Q_f\ge 0$ for these symmetric matrices.
For notational simplicity, we only consider constant parameters
for the model. Except for long time behavior in Section
\ref{sec:long}, our analysis and results can be easily extended to
the case of time-dependent parameters.

Define
\begin{align}
&X(t)=
 \begin{bmatrix} X_1(t) \\
  \vdots \\
   X_N(t)
 \end{bmatrix}\in \mathbb{R}^{Nn},\nonumber
\quad W(t)
=\begin{bmatrix}
W_1(t) \\
 \vdots \\
  W_N(t)
\end{bmatrix}\in \mathbb{R}^{Nn_2},\nonumber   \\
&\widehat \mbA = \mbox{diag}[A, \cdots, A]
+{\bf 1}_{n\times n}\otimes
\frac{G}{N}\in\mathbb{R}^{Nn\times Nn},\nonumber \quad \\
&\widehat \mbD =\mbox{diag}[D, \cdots, D]\in \mathbb{R}^{Nn\times Nn_2},\nonumber \\
&\mbB_k = e_k^N\otimes B\in\mathbb{R}^{Nn\times n_1}, \qquad 1\le k\le N. \nonumber
\end{align}

Now we write system of SDEs in \eqref{stateXi} in the form
\begin{align}\label{bigx}
dX(t)=\Big(\widehat \mbA X(t)+\sum_{k=1}^N \mbB_k u_k(t)\Big)dt+\widehat \mbD dW(t).
\end{align}
Under closed-loop perfect state (CLPS) information, we denote the value function of ${\cal A}_i$  by $V_i(t,\mbx)$, $1\le i\le N$, which corresponds to the initial condition $X(t)= \mbx=(x_1^T, \ldots, x_N^T)^T$ and a cost evaluated on $[t,T]$ in place of \eqref{costJi}.
The set of value functions is determined by the system of  HJB equations
\begin{align}\label{DE}
&0= \frac{\partial V_i}{\partial t}+\min_{u_i\in {\mathbb R}^{n_1}}\Bigg(\frac{\partial^T V_i}{\partial \mbx }\big(\widehat {\mbA}\mbx+\sum_{k=1}^N {\mbB}_k u_k\big)
+u_i^T R u_i \nonumber \\
   &\qquad+|x_i-\Gamma x^{(N)}-\eta|_Q^2  +\frac{1}{2}\mbox{Tr}\big
   ({\widehat{\mbD}^T (V_i)_{\mbx\mbx} \widehat{\mbD}}\big)\Bigg), \\
&V_i(T,\mbx)=|x_i-\Gamma_f x^{(N)}-\eta_f|_{Q_f}^2, \qquad  1\le i\le N, \nonumber
\end{align}
where $x^{(N)}=(1/N)\sum_{k=1}^N x_k$ and the minimizer is
 \begin{align}
 u_i=-\frac{1}{2} R^{-1} {\mbB}_i^T \frac{\partial V_i}
 {\partial \mbx},\qquad 1\le i\le N. \label{uminx}
 \end{align}
   Next we substitute \eqref{uminx} into \eqref{DE} to obtain
\begin{align}\label{DE2}
0 = &\frac{\partial V_i}{\partial t}+\frac{\partial^T V_i}{\partial \mbx}\big({\widehat \mbA}\mbx
-\sum_{k=1}^N \frac{1}{2} {\mbB}_k R^{-1} {\mbB}_k^T \frac{\partial V_k}{\partial \mbx}\big) \nonumber \\
&+|x_i-\Gamma x^{(N)}-\eta|_Q^2  \nonumber \\
   &+\frac{1}{4}\frac{\partial^T V_i}{\partial \mbx}{\mbB}_i R^{-1} \mbB_i^T\frac{\partial V_i}{\partial \mbx}+\frac{1}{2}\mbox{Tr}\big({{\widehat \mbD}^T (V_i)_{\mbx\mbx} {\widehat \mbD}}\big).
\end{align}

Denote
\begin{align*}
&\mbK_i =[0,\cdots,0, {I}_n,0,\cdots,0]-\frac{1}{N}
[\Gamma,\Gamma,\cdots,\Gamma],\\
& \mbK_{if} =[0,\cdots,0, {I}_n,0,\cdots,0]-\frac{1}{N}
[\Gamma_f,\Gamma_f,\cdots,\Gamma_f],\\
& \mbQ_i=\mbK_i^T Q \mbK_i, \quad  \mbQ_{if}=\mbK_{if}^T Q_f \mbK_{if},
\end{align*}
where $I_n$ is the $i$th submatrix.
We write
\begin{align}\label{sym}
|x_i-\Gamma x^{(N)}-\eta|_Q^2  = &\mbx^T {\mbQ}_i \mbx - 2\mbx^T
\mbK_i^TQ\eta
 + \eta^T Q\eta,
    \end{align}
and write $|x_i-\Gamma_f x^{(N)}-\eta_f|_{Q_f}^2 $   in a similar form.

Suppose $V_i(t,\mbx)$ has the following form
\begin{align}\label{Vform}
&V_i(t,\mbx)=\mbx^T {\mbP}_i(t) \mbx+2{\mbS}_i^T(t) \mbx+\mbr_i(t),
\end{align}
where $\mbP_i$ is symmetric.
Then
\begin{align}  \label{1st}
\frac{\partial V_i}{\partial \mbx}=2{\mbP}_i(t)\mbx+2{\mbS}_i(t),\quad
\frac{\partial^2 V_i}{\partial \mbx^2}=2\mbP_i(t).
\end{align}

We substitute \eqref{Vform} and \eqref{1st}  into \eqref{DE2} and derive the equation systems:
\begin{align}\label{DE3_P}
\begin{cases}
\dot{\mbP}_i(t) =  - \Big({\mbP}_i(t)\widehat{\mbA}+\widehat{\mbA}^T
\mathbb{\mbP}_i(t)\Big)+\\
                          \qquad\qquad     \Big({\mbP}_i(t)\sum_{k=1}^N
                             {\mbB}_k R^{-1} {\mbB}_k^T {\mbP}_k(t)\\
                      \qquad\qquad         +\sum_{k=1}^N {\mbP}_k(t){\mbB}_k R^{-1} {\mbB}^T_k {\mbP}_i(t)\Big) \\
                         \qquad\qquad      - {\mbP}_i(t){\mbB}_i R^{-1} {\mbB}_i^T
                             {\mbP}_i(t)- {\mbQ}_i   , \\
 {\mbP}_i(T) = {\mbQ}_{if},
 \end{cases}
\end{align}

\begin{align}\label{DE3_S}
\begin{cases}
\dot{{\mbS}}_i(t) = -\widehat {\mbA}^T {\mbS}_i(t)  -
{\mbP}_i(t){\mbB}_i R^{-1} {\mbB}_i^T {\mbS}_i(t)\\
                            \qquad\qquad    +{\mbP}_i(t) \sum_{k=1}^N {\mbB}_k R^{-1} {\mbB}_k^T {\mbS}_k(t)\\
                             \qquad\qquad   + \sum_{k=1}^N {\mbP}_k(t){\mbB}_k R^{-1}{\mbB}^T_k {\mbS}_i(t)  \\
       \qquad\qquad   +\mbK_i^T Q\eta , \\
{\mbS}_i(T)= -\mbK_{if}^T Q_f\eta_f,
\end{cases}
\end{align}

\begin{align}\label{DE3_gamma}
\begin{cases}
\dot{\mbr}_i(t) =  2{{\mbS}_i}^T(t)\sum_{k=1}^N {\mbB}_k R^{-1}
{\mbB}_k^T {\mbS}_k(t) \\
                   \qquad\qquad  - {\mbS}_i^T(t){\mbB}_i R^{-1} {\mbB}_i^T{\mbS}_i(t)\\
                    \qquad\qquad         - \eta^T Q \eta   -\mbox{Tr}\big(\widehat{\mbD}^T
                          {\mbP}_i(t)\widehat{\mbD}\big), \\
 \mbr_i(T)=\eta_f ^T Q_f \eta_f.
\end{cases}
\end{align}

\begin{remark}\label{remark:P}
If \eqref{DE3_P} has a solution $(\mbP_1, \cdots, \mbP_N)$  on $[\tau,T]\subseteq [0,T]$, such a solution is unique due to the local Lipschitz continuity of the vector field \cite{H69}. Taking transpose on both sides of \eqref{DE3_P} gives an ODE system for $\mbP_i^T$, $1\le i\le N$, which shows that $(\mbP_1^T, \cdots, \mbP_N^T)$ still satisfies \eqref{DE3_P}. So the ODE system \eqref{DE3_P}  guarantees  each $\mbP_i$ to be symmetric
\end{remark}

\begin{remark} \label{remark:PSr}
 If \eqref{DE3_P} has a unique
solution $(\mbP_1,\cdots ,\mbP_N)$ on $[0, T]$, then we can uniquely solve $(\mbS_1, \cdots, \mbS_N)$ and $(\mbr_1, \cdots, \mbr_N)$ by using  linear ODEs.
\end{remark}

For the $N$-player  Nash  game, we consider CLPS information, so that the state vector $X(t)$ is available to each player.

\begin{theorem}\label{theorem:Nash}
Suppose that \eqref{DE3_P} has a unique solution $(\mbP_1,\cdots ,\mbP_N)$ on $[0,T]$. Then we can uniquely solve \eqref{DE3_S}--\eqref{DE3_gamma}, and the game of $N$ players has a set of feedback Nash  strategies given by
$$
u_i=-R^{-1} \mbB_i^T (\mbP_i X(t) +\mbS_i), \qquad 1\le i\le N.
$$
\end{theorem}
\begin{proof} This theorem follows the standard results in \cite[Theorem 6.16, Corollaries 6.5 and 6.12]{BO99}.
\end{proof}

By Theorem \ref{theorem:Nash},   the solution of the feedback Nash strategies  completely reduces to the study  of \eqref{DE3_P}.
For this reason, our subsequent analysis starts
 by analyzing \eqref{DE3_P}.

\section{Asymptotic Solvability}
\label{sec:asy}

\begin{definition}\label{definition:as0}
The sequence of  Nash games \eqref{stateXi}--\eqref{costJi} with  closed-loop perfect state information  has asymptotic solvability if there exists $N_0$ such that for all $N\ge N_0$,
 $(\mbP_1,\cdots,\mbP_N)$ in \eqref{DE3_P} has a solution on $[0,T]$ and
\begin{align}\label{main_conl1}
\sup_{N\ge N_0}\sup_{1\le i\le N, 0\leq t\leq T}  \|\mbP_i(t)\|_{l_1} <\infty.
\end{align}
\end{definition}

 Definition \ref{definition:as0} only involves the Riccati equations. This is sufficient due to Remark \ref{remark:PSr}. The boundedness condition
\eqref{main_conl1} is to impose certain regularity of the solutions, which is necessary for
studying the asymptotic behavior of the system when $N\to \infty$.

Let the $Nn\times Nn$ identity matrix be partitioned in the form:
\begin{align*}
I_{Nn}=\begin{bmatrix}
I_n   &0    &\cdots        &0    \\
0  &I_n    &\cdots         &0     \\
\vdots  &\vdots     &\ddots        &\vdots     \\
0  &0    &0         &I_n
\end{bmatrix}.
\end{align*}
For $1\le i\ne j\le N$,
exchanging the $i$th and $j$th rows of submatrices in $I_{Nn}$,  let $J_{ij}$ denote the resulting matrix. For instance, we have
\begin{align*}
J_{12}=\begin{bmatrix}
0   &I_n    &\cdots        &0    \\
I_n  &0    &\cdots         &0     \\
\vdots  &\vdots     &\ddots        &\vdots     \\
0  &0    &0         &I_n
\end{bmatrix}.
\end{align*}
It is easy to check that $J_{ij}^T=J_{ij}^{-1}=J_{ij}$.

\begin{theorem}\label{theorem:Prep3}
 We assume that \eqref{DE3_P} has  a solution
 $(\mbP_1(t), \cdots,
\mbP_N(t))$  on $[0,T]$. Then the following holds.

i) ${\mbP}_1(t)$  has the  representation
\begin{align}\label{P_matrix3}
{\mbP}_1(t)=\begin{bmatrix}
\Pi_1(t) &\Pi_2(t) &\Pi_2(t)&\cdots &\Pi_2(t) \\
\Pi_2^T(t) &\Pi_3(t) &\Pi_3(t)&\cdots &\Pi_3(t)\\
\Pi_2^T(t) &\Pi_3(t)&\Pi_3(t) &\cdots &\Pi_3(t)\\
\vdots          & \vdots        &  \vdots        &\ddots &\vdots \\
\Pi_2^T(t) &\Pi_3(t) &\Pi_3(t)&\cdots &\Pi_3(t)
\end{bmatrix},
\end{align}
where $\Pi_1$ and $\Pi_3$ are $n\times n$ symmetric matrices.

ii) For $i>1$,  $\mbP_i(t)= J_{1i}^T \mbP_1(t) J_{1i}$.
\end{theorem}

\begin{proof} See Appendix A. \end{proof}

By Theorem \ref{theorem:Prep3}, \eqref{main_conl1} is equivalent to the following condition:
\begin{align}\label{main_con2}
\sup_{N\ge N_0, 0\leq t\leq T} \left(|\Pi_1(t)|+N|\Pi_2(t)|+N^2|\Pi_3(t)|\right)<\infty.
\end{align}

We present some continuous dependence result of parametrized  ODEs in
 Theorem \ref{theorem:depen} below.
 This will play a  key role in establishing Theorem \ref{theorem:iff} later.
Consider
\begin{align}
&\dot{x}=f(t,x), \quad x(0)=z\in\mathbb{R}^K, \label{dot_x} \\
&\dot{y}=f(t,y)+g(\epsilon, t,y), \label{dot_y}
\end{align}
where $y(0)=z_{\epsilon}\in\mathbb{R}^K$, $0<\epsilon\leq 1.$

Let $\phi(t,x)=f(t,x),$ or $f(t,x)+g(\epsilon, t, x)$.
We introduce the following assumptions on \eqref{dot_x} and \eqref{dot_y}.

(A1) $\sup_{\epsilon, 0\leq t\leq T} |f(t,0)|+|g(\epsilon, t, 0)|\leq C_1$.

(A2) $\phi(\cdot, x)$ is Lebesgue measurable for each fixed $x\in\mathbb{R}^K$.

(A3) For each $t\in [0, T]$, $\phi(t,x): \mathbb{R}^K\rightarrow\mathbb{R}^K$ is locally Lipschitz in $x$, uniformly with respect to $(t,\epsilon)$, i.e., for any fixed $r>0$, and $x, y\in B_{r}(0)$ which is the open ball of radius $r$ centering $0$,
\begin{align*}
|\phi(t, x)-\phi(t, y)|\leq \mbox{Lip} (r)|x-y|,
\end{align*}
where $\mbox{Lip}(r)$ depends only on $r$, not on $\epsilon\in (0,1],t\in [0,T]$.

(A4)   $\lim_{\epsilon\rightarrow 0}|z_{\epsilon}-z|=0$, and for each fixed $r>0$,
\begin{align*}
\lim_{\epsilon\rightarrow 0}\sup_{0\leq t\leq T, y\in B_r(0)} |g(\epsilon, t, y)|=0.
\end{align*}

 If the solutions to \eqref{dot_x} and \eqref{dot_y}, denoted by $x^z(t)$
 and $y^\epsilon(t)$,  exist on $[0,T]$, they are unique by the local Lipschitz condition (A3); in this case denote
$\delta_\epsilon = \int_0^T |g(\epsilon, \tau, x^z(\tau))|d\tau$, which converges to $0$ as $\epsilon\to 0$ due to (A4).
\begin{theorem} \label{theorem:depen}
Under Assumptions (A1)--(A4), we have the following assertions:

i) If \eqref{dot_x} has a solution $x^z(t)$ on $[0, T]$, then there exists $0<\bar{\epsilon}\leq 1$ such that for all $0<\epsilon\leq\bar{\epsilon}$, \eqref{dot_y} has a solution $y^\epsilon(t)$ on $[0, T]$ and
\begin{align}
\sup_{0\le t\le T}|y^\epsilon (t)-x^z(t)|=O(|z_\epsilon-z|
+\delta_\epsilon). \label{oe}
\end{align}

ii) Suppose there exists a sequence $\{\epsilon_i, i\geq 1\}$ where $ 0<\epsilon_i\leq 1$ and $\lim_{i\to \infty}\epsilon_i= 0$ such that \eqref{dot_y} with $\epsilon=\epsilon_i$ has a solution $y^{\epsilon_i}$   on $[0, T]$   and
$\sup_{i\ge 1, 0\leq t\leq T} |y^{\epsilon_i}(t)|\leq C_2$
for some constant $C_2$. Then \eqref{dot_x} has a solution on $[0, T]$.
\end{theorem}

\begin{proof}
See Appendix B. \end{proof}

\begin{remark} \label{remark:tc}
If \eqref{dot_x} and \eqref{dot_y} are replaced by matrix ODEs and (or) a terminal condition at $T$ is used in each equation, the results in Theorem \ref{theorem:depen} still hold.
\end{remark}

 Let $$M=BR^{-1}B^T.$$ Before presenting further results, we introduce two Riccati ODEs:
\begin{align}\label{d11}
\begin{cases}
\dot{\Lambda}_1 =  \Lambda_1M \Lambda_1-(\Lambda_1A+A^T\Lambda_1)-Q, \\
\Lambda_1(T)=Q_f,
\end{cases}
\end{align}
and
\begin{align}\label{d21}
\begin{cases}
\dot{\Lambda}_2 = \Lambda_1M \Lambda_2+ \Lambda_2 M \Lambda_1+ \Lambda_2M\Lambda_2 \\
                             \quad \qquad - (\Lambda_1G + \Lambda_2 (A+G) +A^T\Lambda_2) +Q\Gamma, \\
 \Lambda_2(T)=-Q_f\Gamma_f. 
 \end{cases}
\end{align}
Note that \eqref{d11} is the standard Riccati ODE in LQ optimal control and has a unique solution $\Lambda_1$ on $[0,T]$. Equation
\eqref{d21} is a non-symmetric Riccati ODE, where $\Lambda_1$ is now treated as a known function. We state the main theorem on asymptotic solvability.

\begin{theorem}\label{theorem:iff}
The sequence of games  in  \eqref{stateXi}--\eqref{costJi} has asymptotic solvability if and only if   \eqref{d21} has a unique solution on $[0, T]$.
\end{theorem}

\begin{proof}
See Appendix  C.
\end{proof}

We outline the key idea for identifying this necessary and sufficient condition of asymptotic solvability. By Theorem \ref{theorem:Prep3} and the ODE of $\mbP_1(t)$ in \eqref{DE3_P}, we obtain an ODE system of the form
$$
\begin{bmatrix}
{\dot\Pi}_1\\
{\dot\Pi}_2\\
{\dot\Pi}_3
\end{bmatrix}
 = \Psi_N(\Pi_1, \Pi_2, \Pi_3).
$$
However, directly taking $N\to \infty$ is not useful because  this method
on one hand will not
 generate a meaningful  limit of the vector field $\Psi_N$ owing to terms such as $(N-1) \Pi_2 M\Pi_2$ in $\Psi_N$ (see \eqref{d1} ) and on the other will  cause a loss of
dynamical information since $(\Pi_2, \Pi_3)$ can vanish when $N\to\infty$.
Our method is to re-scale by
defining
\begin{align} \label{new_system_1}
&\Lambda_1^N=\Pi_1(t), \ \Lambda_2^N=N\Pi_2(t),
\ \Lambda_3^N=N^2\Pi_3(t),  
\end{align}
and examine their ODE system.
This procedure leads to a new limiting
ODE system which can preserve key information about the dynamics of  $ (\Pi_1, \Pi_2, \Pi_3)$
and which consists of  \eqref{d11} and \eqref{d21} together with another equation:
\begin{align} \label{d3_1}
\begin{cases}
\dot{\Lambda}_3 = \Lambda_2^TM \Lambda_2 + \Lambda_3M \Lambda_1
+ \Lambda_1M \Lambda_3+\Lambda_3M \Lambda_2+\Lambda_2^TM\Lambda_3 \\
                             \qquad -
                             \left(\Lambda_2^T G+G^T\Lambda_2
                             +\Lambda_3( A+G)
                             +(A^T+G^T)\Lambda_3 \right)  \\
 \qquad - \Gamma^T Q\Gamma,  \\
  \Lambda_3(T)=\Gamma_f^T Q_f\Gamma_f.
\end{cases}
\end{align}
 Note that after \eqref{d11} and \eqref{d21} are solved on $[0,T]$ (or
 otherwise on a maximal existence interval for the latter), \eqref{d3_1}  becomes  a linear ODE.

\begin{theorem}\label{theorem:PiLa}
Suppose \eqref{d21} has  a solution on $[0,T]$. Then we have
\begin{align}
\sup_{0\le t\le T}(|\Pi_1-\Lambda_1|  +| N\Pi_2 -
\Lambda_2|+|N^2\Pi_3-\Lambda_3|)=O(1/N).\nonumber
\end{align}
\end{theorem}

\begin{proof}  The bound follows from Theorem \ref{theorem:depen} i) by use of
 $g_1, g_2, g_3$ and  the terminal
 conditions which appear in the equations of  $\Lambda_1^N$, $\Lambda_2^N$, $\Lambda_3^N$ in Appendix C.  \end{proof}

\subsection{Decentralized Control}
\label{sec:dec}

\begin{proposition} \label{prop:sr}
Assume that \eqref{DE3_P} has  a solution
 $(\mbP_1, \cdots,
\mbP_N)$  on $[0,T]$.
Then the assertions hold:

i)  $\mbS_i(t)$ in \eqref{DE3_S} has the  form
\begin{align}\label{S_form}
\mbS_i(t)=[ \theta^T_2(t), \cdots,\theta_1^T(t),\cdots,
\theta_2^T(t)
]^T,
\end{align}
in which the $i$th sub-vector is $\theta_1(t)\in \mathbb{R}^n$ and the remaining sub-vectors are $\theta_2(t)\in \mathbb{R}^n$.

ii) Furthermore, $\mbr_1=\mbr_2=\cdots=\mbr_N$ for $t\in[0,T]$.
\end{proposition}

\begin{proof} See Appendix C. \end{proof}

We introduce  two  ODEs:
\begin{align}\label{theta_1_new}
\begin{cases}
\dot{\chi}_1(t) =  (\Lambda_1M+\Lambda_2M-A^T){\chi_1} + Q\eta,
\\
 {\chi_1}(T)=-Q_f\eta_f,
\end{cases}
\end{align}
and
\begin{align}\label{theta_2_new}
\begin{cases}
\dot{\chi}_2(t) = ((\Lambda_2^T+\Lambda_3)M-G^T){\chi_1}  \\
\qquad\quad+((\Lambda_1+\Lambda_2^T)M
                        -(A^T+G^T)){\chi_2}-\Gamma^TQ\eta,
                         \\
     {\chi_2}(T)=\Gamma_f^T Q_f\eta_f.
\end{cases}
\end{align}

 Define
\begin{align}
{\chi}_1^N(t)=\theta_1(t),\quad {\chi}_2^N(t)=N\theta_2(t). \label{ch12}
\end{align}
In fact \eqref{theta_1_new} and \eqref{theta_2_new} can be derived as the limit of the ODEs satisfied by $(\chi_1^N, \chi_2^N)$; see Appendix  C.

\begin{proposition} \label{prop:sichi}
For $(\theta_1(t), \theta_2(t))$ specified in \eqref{S_form}, we have
\begin{align}
\sup_{0\le t\le T}(|\theta_1(t)- \chi_1(t)|
+|N\theta_2(t)-\chi_2(t)|) =O(1/N). \label{sichi}
\end{align}
\end{proposition}

\begin{proof} See Appendix C. \end{proof}

By Theorem \ref{theorem:Nash}, the strategy of player ${\cal A}_i$ is
\begin{align}\label{uxN}
u_i= -R^{-1} B^T\Big(\Pi_1(t)X_i+\Pi_2(t)\sum_{j\ne i} X_j +\theta_1(t)\Big).
\end{align}
The closed-loop equation of $X_i$ is  now given by
\begin{align}
&dX_i(t)= \Big(AX_i-M \Big(\Pi_1X_i+\Pi_2\sum_{j\ne i}
X_j +\theta_1\Big) \nonumber \\
  & \qquad\qquad
 +G X^{(N)}\Big)dt + D dW_i, \nonumber
\end{align}
which gives
\begin{align}
dX^{(N)} =&[ (A-M(\Pi_1 +(N-1)\Pi_2)+G)X^{(N)} -M \theta_1] dt
\nonumber \\
&+ \frac{D}{N}\sum_{i=1}^N dW_i. \label{xndw}
\end{align}
To denote the limit of \eqref{xndw} when $N\to \infty$,  we introduce
the closed-loop mean field dynamics
\begin{align}\label{clxbar}
\frac{d\bar X}{dt}= \left(A-M (\Lambda_1 +\Lambda_2)+G\right)\bar X -M \chi_1(t),
\end{align}
where $\Bar X(0)= x_0$.
 \begin{proposition}\label{prop:err}
Suppose  $E\sup_{i\ge 1}|X_i(0)|^2\le C$ for some fixed constant $C$ and
$\lim_{N\to \infty} \frac{1}{N} \sum_{i=1}^N EX_i(0)= x_0$.
Then
$$\sup_{0\le t\le T} E| X^{(N)}(t)- \bar X(t)|^2 =
O(|\frac{1}{N} \sum_{i=1}^N EX_i(0)-x_0|^2+1/N ). $$
\end{proposition}

\begin{proof} By \eqref{xndw}--\eqref{clxbar}, we find the explicit expression of $X^{(N)}(t)-\bar X(t)$. The proposition follows from elementary estimates by use of Theorem \ref{theorem:PiLa} and   Proposition~\ref{prop:sichi}.~\end{proof}

When $N\to \infty$, from \eqref{uxN} we obtain the  control law
\begin{align}
u_i^d= -R^{-1} B^T \left(\Lambda _1 X_i+ \Lambda_2  \bar X
+\chi_1(t)\right), \label{ud1}
\end{align}
which is decentralized since $\bar X $ and $\chi_1$ do not depend on the sample path information of other players and can be computed off-line.
Suppose $\Lambda_1$ and $\Lambda_2$ have been given on $[0,T]$.
Then    \eqref{ud1} can be determined by solving
the decoupled ODE system \eqref{theta_1_new} and \eqref{clxbar},
which has a unique solution. Note that \eqref{theta_1_new} has its origin in dynamic programming.

\section{Relation to the Fixed Point Approach}

\label{sec:fp}

The fixed point approach for solving the LQ mean field game consists of two steps (see e.g. \cite{HCM07}).

 Step 1. We use $\overline X \in C([0,T], \mathbb{R}^n)$ to
 approximate $X^{(N)}$ in \eqref{stateXi}
and consider the optimal control problem with dynamics and cost:
\begin{align}
&dX_i^\infty(t) = (AX_i^\infty(t)+Bu_i(t) +G \overline X(t))dt  +D dW_i, \label{xinfi} \\
&\bar J_i(u_i)= E\int_0^T  \left(
|X_i^\infty-\Gamma  \overline X -\eta|_Q^2 +u_i^TRu_i\right) dt \nonumber\\
 &\qquad\qquad +E|X_i^\infty(T)-\Gamma_f
 \overline X(T) -\eta_f|_{Q_f}^2,\nonumber
\end{align}
where we set $X_i^\infty(0)=X_i(0)$. The Brownian motion is the same as in \eqref{stateXi}.
Applying dynamic programming,  the optimal control law is given by
\begin{align}
\hat u_i= -R^{-1} B^T (\Lambda_1 X_i^\infty(t) +s(t)), \label{uihat}
 \end{align}
where $\Lambda_1$ is solved from \eqref{d11} and
$$
\dot{s}(t) = -(A^T-\Lambda_1M)s(t) -\Lambda_1 G\overline X(t)
+ {Q} (\Gamma \overline X(t) +\eta ),
$$
and $s(T)=- Q_f(\Gamma_f\overline X(T) +\eta_f)$.

Step 2. Let $\lim_{N\to \infty} \frac{1}{N}\sum_{i=1}^N EX_i^\infty(t)$ be determined from
the closed-loop system of \eqref{xinfi} under the control law $\hat u_i$
and the given $\overline X$. By the standard consistency requirement in mean field games \cite{HCM07}, we impose
$  \overline X(t) =\lim_{N\to \infty} \frac{1}{N}\sum_{i=1}^N EX_i^\infty(t)$ for all $t\in [0,T]$, which amounts to  specifying $\overline X$ as a fixed point.
This introduces the equation
$$
\frac{d\overline X }{ dt} = (A -M \Lambda_1 +G)\overline X - M   s,
$$
where $\overline X(0)= x_0$ and we assume $\lim_{N\to \infty} \frac{1}{N} \sum_{i=1}^N EX_i(0)= x_0$ as in Section \ref{sec:asy}.

Combining the  ODEs of $s$ and $\overline X$  gives
the MFG solution equation system
\begin{align}\label{mfgfp}
\begin{cases}
\frac{ d\overline X}{dt} = (A -M \Lambda_1+G)\overline X
- M   s ,   \\
 \dot{s}= -(A^T-\Lambda_1 M)s -\Lambda_1 G\overline X  + {Q} (\Gamma \overline X +\eta ),
\end{cases}
\end{align}
where $\overline X(0)=x_0$ and $s(T)=- Q_f(\Gamma_f\overline X(T) +\eta_f)$.
The equation system \eqref{mfgfp}
 is a TPBV problem.

\begin{remark}
We introduce in \eqref{mfgfp} the new notation $\overline X$
 instead of $\bar X$. It is necessary to maintain this distinction since the two functions coincide only under certain conditions as shown later.  \end{remark}

\subsection{Solving the TPBV Problem  }

\label{sec_10_MFG}

Denote
\begin{align} \label{mbbA}
\mathbb{A}(t) = \begin{bmatrix} A-M\Lambda_1(t) +G & -M \\
Q\Gamma-\Lambda_1(t) G & -A^T+\Lambda_1(t)M
\end{bmatrix}.
\end{align}
The fundamental solution matrix of \eqref{mfgfp} is determined by
 the matrix ODE
\begin{align}
\frac{\partial}{\partial t}\Phi(t, \tau)=\mathbb{A} \Phi(t, \tau), \quad \Phi(\tau, \tau)=I_{2n}. \label{phiAode}
\end{align}
Denote
\begin{align}\label{fundsol}
\Phi(t,\tau) =\begin{bmatrix} \Phi_{11}(t, \tau) & \Phi_{12} (t, \tau) \\
\Phi_{21}(t, \tau) & \Phi_{22}(t, \tau)
\end{bmatrix},
\end{align}
where each submatrix is $n\times n$.

Denote
\begin{align*}
&Z_1= \Phi_{22}(T,0)+Q_f \Gamma_f \Phi_{12}(T,0) \in \mathbb{R}^{n\times n},\\
& Z_2= [\Phi_{21}(T,0)+Q_f\Gamma_f \Phi_{11}(T,0)]x_0 + Q_f \eta_f\\
&\qquad + \int_0^T [\Phi_{22}(T,\tau)+Q_f\Gamma_f\Phi_{12}(T,\tau)]
Q\eta d\tau \in \mathbb{R}^{n}.
\end{align*}

\begin{proposition}\label{prop_1}
i) \eqref{mfgfp} has a solution if and only if
\begin{align}
Z_2\in \mbox{span} \{Z_1\}. \label{z12}
\end{align}

ii) If $\mbox{det}\big(Z_1\big)\neq 0$, \eqref{mfgfp} has a unique solution.
\end{proposition}

\begin{proof} i) We introduce $s(0)$ to be determined. By \eqref{mfgfp},
\begin{align*}
\begin{bmatrix} \overline X(T) \\ s(T) \end{bmatrix} =
& \Phi(T, 0)\begin{bmatrix} x_0 \\ s(0)\end{bmatrix} + \int_0^T \Phi(T, \tau)
\begin{bmatrix} 0 \\ Q\eta \end{bmatrix} d\tau \\
 =& \begin{bmatrix} \Phi_{11}(T,0) &\Phi_{12}(T,0) \\ \Phi_{21}(T,0)
 &\Phi_{22}(T,0) \end{bmatrix} \begin{bmatrix} x_0 \\ s(0) \end{bmatrix}\\
& + \int_0^T \begin{bmatrix} \Phi_{11}(T, \tau) &\Phi_{12}(T, \tau) \\ \Phi_{21}(T, \tau) & \Phi_{22}(T, \tau)\end{bmatrix} \begin{bmatrix} 0 \\ Q\eta \end{bmatrix} d\tau.
\end{align*}

Then \eqref{mfgfp} has a solution if and only if there exists $s(0)$ such that
\begin{align}\label{sTs0}
s(T)&=\Phi_{21}(T, 0)x_0 +\Phi_{22}(T, 0)s(0)+
\int_0^T \Phi_{22}(T, \tau)Q\eta d\tau,\nonumber\\
&= - Q_f(\Gamma_f\overline X(T) +\eta_f),
\end{align}
which is equivalent to
$Z_1 s(0)+ Z_2=0$.
This proves part i).

ii) This part follows from part i).
\end{proof}

For illustration, we consider the special case
with $\Gamma_f=0$, $\eta_f=0$. Then \eqref{z12} in Proposition \ref{prop_1} i)
becomes
\begin{align*}
&\Phi_{21}(T, 0) x_0 +\int_0^T \Phi_{22}(T, \tau) Q\eta d\tau
\in \mbox{span}\big\{\Phi_{22}(T, 0)\big\}.
\end{align*}

\subsection{Direct Approach Solvability Implies Fixed Point Solvability}
\begin{theorem} \label{theorem:astoda}
Suppose $\Lambda_2$ has a solution  on $[0,T]$. Then
the following holds.

i)
\eqref{mfgfp} has a unique solution $(\overline X, s)$ given by
$$
\begin{cases}
\overline X(t)= \bar X(t),\\
 s(t)= \Lambda_2(t) \bar X(t) +\chi_1(t),
\end{cases}
 $$
where $(\bar X, \chi_1)$ is solved from \eqref{theta_1_new} and
\eqref{clxbar} in  the direct  approach.

ii) Asymptotic solvability of the sequence of games \eqref{stateXi}--\eqref{costJi} implies that \eqref{mfgfp} has a unique solution.
\end{theorem}

\begin{proof} i) For \eqref{mfgfp}, we write
\begin{align}
s= \Lambda_2\overline X+ \varphi(t), \label{sL2p}
\end{align}
where $\varphi$ is a new unknown function.
Now \eqref{mfgfp} is transformed into a new equation system in terms of $(\overline X, \varphi)$, where
$$
\dot\varphi = (\Lambda_1 M +\Lambda_2M  -A^T)\varphi +Q\eta,\quad \varphi(T)= -Q_f\eta_f.
$$
The terminal condition $\varphi(T)$ has been determined from
\eqref{sL2p} with $t=T$.
We can uniquely solve $\varphi$ and in fact $\varphi =\chi_1$. Subsequently, we further obtain $\overline X=\bar X$. It is clear the solution $(\overline X, s)$ is unique.

ii) This part follows from  Theorem
\ref{theorem:iff} and part i).
\end{proof}

Let \eqref{uihat} be applied by the $N$ players in \eqref{stateXi}, and accordingly denote
\begin{align}
\hat u_i^d= -R^{-1} B^T (\Lambda_1 X_i(t) +s(t)). \label{uhdi}
\end{align}
Under the asymptotic solvability condition, the two control laws $\hat u_i^d$ in \eqref{uhdi} and $u_i^d$ in \eqref{ud1} are equivalent by Theorem \ref{theorem:astoda}.
Based on assumptions on the initial states as given in Proposition \ref{prop:err}, one can apply the standard method in \cite{HCM07} to show that the set of strategies $(\hat u_1^d, \cdots, \hat u_N^d)$ in \eqref{uhdi} is an $\epsilon$-Nash equilibrium of the $N$-player game, where $\epsilon\to 0$ as $N\to\infty$.

The  existence and uniqueness condition in the TPBV problem
is quite different from the condition for
asymptotic solvability. It is possible that the Riccati equation  of $\Lambda_2$ has a finite escape time in $[0, T)$ but the TPBV problem is still solvable. A detailed comparison will be developed in the next section for scalar models.

\subsection{Fixed Point via A Contraction Mapping}

The original analysis in \cite{HCM07} applies
the fixed point approach to infinite time horizon LQ mean field
games and establishes existence and uniqueness of a solution by
specifying a contraction mapping. The procedure in \cite{HCM07}
can be applied to \eqref{mfgfp} to derive a corresponding contraction condition as well. By Theorem \ref{theorem:astoda}, asymptotic solvability in the direct approach implies the fixed point solvability, but the converse may not hold (and is indeed not true as it turns out later).
Now if the fixed point is determined from a contraction mapping
as in \cite{HCM07}, an intriguing question is
what is its implication regarding asymptotic solvability. Below we show asymptotic solvability holds in this case.

To facilitate further analysis, we consider \eqref{mfgfp} on a general interval  $[t_0, T]$ for $t_0\in [0,T)$,  and rewrite it as below:
\begin{align}\label{mfgt0}
\begin{cases}
\frac{ d\overline X}{dt} = (A -M \Lambda_1+G)\overline X
- M   s ,   \\
 \dot{s}= -(A^T-\Lambda_1 M)s -\Lambda_1 G\overline X  + {Q} (\Gamma \overline X +\eta ).
\end{cases}
\end{align}
The initial and terminal  conditions are given by $\overline X(t_0)=x_{t_0}$ and $s(T)=- Q_f(\Gamma_f\overline X(T) +\eta_f)$.

Denote the linear ODEs
\begin{align*}
&\dot y_1= (A-M \Lambda_1(t) +G)y_1,\quad
\dot y_2= (-A^T+ \Lambda_1(t)M)y_2,
\end{align*}
where $t\in [0,T]$ and $y_i(t)\in \mathbb{R}^n$.
Let $\Psi_1$ and $\Psi_2$ be their fundamental solution matrices
so that
\begin{align*}
&\frac{\partial \Psi_1 (t, \tau)}{\partial t}
 =(A-M \Lambda_1(t) +G) \Psi_1(t, \tau), \quad \Psi_1(\tau, \tau)=I,\\
&\frac{\partial \Psi_2 (t, \tau)}{\partial t}
 =(-A^T+ \Lambda_1(t) M ) \Psi_2(t, \tau), \quad \Psi_2(\tau, \tau)=I.
\end{align*}

 Following
the procedure in \cite{HCM07}, we solve $s$
from the second equation of \eqref{mfgt0} to obtain
\begin{align}
s(t)=& -\Psi_2(t,T) Q_f \Gamma_f \overline X(T)\nonumber\\
&-\int_t^T \Psi_2(t,r)
(Q \Gamma - \Lambda_1(r) G) \overline X(r)dr 
+\zeta_1(t), \label{sz1}
\end{align}
where $\zeta_1$ depends on $(\eta, \eta_f)$ but not on $\overline X$.
Substituting \eqref{sz1} into
 the first equation of \eqref{mfgt0}, we have the expression
\begin{align}
\overline X(t)=& \Psi_1(t, t_0) x_{t_0}+\int_{t_0}^t \Psi_1(t, \tau) M \Psi_2(\tau, T) Q_f \Gamma_f \overline X(T)d \tau \nonumber \\
  & + \int_{t_0}^t \int_\tau^T \Psi_1(t, \tau)M \Psi_2(\tau, r) (Q\Gamma-\Lambda_1(r)G)\overline X(r)dr d\tau \nonumber \\
&+\zeta_2(t),\label{xodet0}
\end{align}
where $\zeta_2$ depends on $(\eta, \eta_f)$ but not on $(\overline X, x_{t_0})$.
Denote the operator
$\Upsilon_{t_0}$: $C([t_0, T], \mathbb{R}^n) \to C([t_0,T], \mathbb{R}^n)$  as follows:
\begin{align*}
(\Upsilon_{t_0} \phi)(t) &=\int_{t_0}^t \Psi_1(t, \tau) M \Psi_2(\tau, T) Q_f \Gamma_f \phi(T)d \tau \\
  & + \int_{t_0}^t \int_\tau^T \Psi_1(t, \tau)M \Psi_2(\tau, r) (Q\Gamma-\Lambda_1(r)G)\phi(r)dr d\tau.
\end{align*}
We
take the norm $\|\phi\|=\sup_{t\in[t_0,T]}|\phi(t)|$ in $C([t_0, T], \mathbb{R}^n)$.  Now  \eqref{xodet0} can be written as
$$
\overline X(t)=\Psi_1(t, t_0) x_{t_0} +(\Upsilon_{t_0} \overline X)(t) +\zeta_2(t), \qquad t\in [t_0, T].
$$

Denote the constant
\begin{align}
\kappa_{t_0} = &\sup_{t_0\le t\le T} \Big[ \int_{t_0}^t \int_\tau^T \left|\Psi_1(t, \tau)M \Psi_2(\tau, r) [Q\Gamma-\Lambda_1(r)G]\right|dr d\tau \nonumber \\
&+\int_{t_0}^t| \Psi_1(t, \tau) M \Psi_2(\tau, T) Q_f \Gamma_f |d \tau \Big]. \nonumber
\end{align}
We have the estimate
$$
\|\Upsilon_{t_0} \phi_1 -\Upsilon_{t_0}\phi_2\|\le  \kappa_{t_0}\|\phi_1-\phi_2\|, \quad \forall\ \phi_1, \phi_2\in C([t_0, T],\mathbb{R}^n).
$$
It is straightforward to check that $\kappa_{t_0}\le \kappa_0$ for all $t_0\in [0,T]$.

\begin{theorem} \label{theorem:conAS}
Suppose $\kappa_0 <1$. Then asymptotic solvability holds for the sequence of games \eqref{stateXi}--\eqref{costJi}.
\end{theorem}

\begin{proof}
We prove by contradiction. Suppose asymptotic solvability doe not hold for \eqref{stateXi}--\eqref{costJi}, which  implies $\Lambda_2$ has a maximal existence interval $(t^*, T]$ for $t^*\in [0, T)$. So there exists a strictly decreasing sequence $\{t_k, k\ge 1\}$ converging to $t^*$ such that $\lim_{k\to \infty}|\Lambda_2(t_k)|=\infty$.
We can find an appropriate subsequence, still denoted by  $\{t_{k}, k\ge 1 \}$, such that for some $(\hat i, \hat j)$, we have
\begin{align}\label{Lhatij}
\lim_{k\to \infty}|\Lambda_2^{\hat i, \hat j} (t_{k})|=\infty,\quad
|\Lambda_2^{\hat i, \hat j} (t_{k})| =\max_{1\le i,j\le n}  |\Lambda_2^{ i,  j} (t_{k})|,
\end{align}
where the superscripts indicate the $(i,j)$-th entry of $\Lambda_2(t_k)$.

Now for $t_k$ in \eqref{Lhatij}, we select $x_{t_{k}}= [0, \ldots, 0, 1, 0, \ldots, 0 ]^T=e^n_{\hat j}$, and solve a special form of  \eqref{mfgt0} on $[t_{k}, T]$
as follows:
\begin{align}\label{mfgtk}
\begin{cases}
\frac{ d\overline X^*}{dt} = (A -M \Lambda_1+G)\overline X^*
- M   s^* ,   \\
 \dot{s}^*= -(A^T-\Lambda_1 M)s^* +( Q\Gamma -\Lambda_1 G)\overline X^*,
\end{cases}
\end{align}
 which has  initial condition $\overline X^*(t_{k})=x_{t_{k}} =e^n_{\hat j}$ and terminal condition $s^*(T)=- Q_f\Gamma_f\overline X^*(T) $. By the relation
$$
\overline X^*(t)=\Psi_1(t, t_{k}) x_{t_{k}}+ (\Upsilon_{t_{k}} \overline X^*)(t) , \qquad t\in [t_{k},T]
$$
and $\kappa_{t_{k}}\le \kappa_0$,
we obtain a unique solution $( \overline  X^*, s^*  )
\in C([t_k, T], \mathbb{R}^{2n})$ for \eqref{mfgtk} and have  the bound
$$
\|\overline X^*\|\le \frac{1}{1-\kappa_0} \sup_{t,\tau\in [0,T]} | \Psi_1(t, \tau)|. $$

In parallel to  \eqref{sz1},
\begin{align}
s^*(t)=& -\Psi_2(t,T) Q_f \Gamma_f \overline X^*(T) \nonumber\\
&-\int_t^T \Psi_2(t,r)
(Q \Gamma - \Lambda_1(r) G) \overline X^*(r)dr.\nonumber
\end{align}
 We may further find a fixed constant $C_0$ independent of $t_{k}$ such that
\begin{align}
\sup_{t\in [t_k, T]}[|\overline X^*(t)|+|s^*(t)|]\le C_0 . \label{xsc0}
\end{align}
On the other hand, for each $t_k$ appearing in \eqref{Lhatij} and the resulting  interval $[t_{k}, T]$, by the fact that $\Lambda_2$ exists on $(t^*, T]\supset[t_k,T]$, we may use the method in Theorem \ref{theorem:astoda} to show the relation
$$
s^*(t)= \Lambda_2(t) \overline X^* (t), \qquad t\in [t_{k}, T].
$$
Hence
$s^*(t_{k}) =\Lambda_2(t_{k})e^n_{\hat j}$,
and by \eqref{Lhatij},
$$
\lim_{k\to \infty}|s^*(t_{k}) | =\infty,
$$
which contradicts  \eqref{xsc0}.

We conclude that $\Lambda_2$ has a solution on $[0,T]$.
Therefore, asymptotic solvability holds for \eqref{stateXi}--\eqref{costJi}.
\end{proof}

\begin{remark}
We use $\kappa_0<1$ to ensure a contraction condition for the TPBV problem defined on $[0,T]$. It is possible to have improved contraction estimates.
Our method here is adequate for addressing the qualitative relation as shown in Theorem~\ref{theorem:conAS}.   
\end{remark}

\section{The Scalar Case: Explicit Solutions}

\label{sec:scal}

\subsection{Riccati Equations of Asymptotic Solvability}

We analyze a scalar case of
 the Riccati ODEs \eqref{d11} and \eqref{d21}, i.e., $n=1$, and
 suppose $B\ne 0$ for the model to be nontrivial.   Consider
\begin{numcases}{}
\dot{\Lambda}_1=\Lambda_1^2-2A\Lambda_1-Q, \label{sric1} \\
\dot{\Lambda}_2=2\Lambda_1\Lambda_2+\Lambda_2^2-\Lambda_2(2A+G)-\Lambda_1 G+Q\Gamma, \label{sric2}
\end{numcases}
where $\Lambda_1(T)=Q_f$ and $\Lambda_2(T)=-Q_f\Gamma_f$.
Without loss of generality we only deal with the case $M=1$ since otherwise  a change of variable may be used to convert \eqref{d11}--\eqref{d21} to the above form with  appropriately modified  parameters $Q$ and $Q_f$.
 Although $\Lambda_1(t)$ can be explicitly  solved for a general $Q_f$, one usually cannot further solve $\Lambda_2(t)$ in a closed form.
 To overcome this difficulty, we will further take particular choices of the terminal conditions to obtain explicit solutions.
Our method is to choose $Q_f$ appropriately to solve $\Lambda_1(t)$
as a constant so that \eqref{sric2} becomes a Riccati equation with constant coefficients.

In this section we further suppose the pair $(A, \sqrt{Q})$ is detectable.
Denote the algebraic Riccati equation
\begin{align*}
\Lambda_{1\infty}^2-2A\Lambda_{1\infty}-Q=0,
\end{align*}
which gives the stabilizing solution
\begin{align}\label{sol1}
\Lambda_{1\infty}=A+\sqrt{A^2+Q}\ge 0,
\end{align}
such that $A-M\Lambda_{1\infty}=A-\Lambda_{1\infty} <0$.  Below we take
\begin{align} \label{QGf}
M=1,\quad  Q_f=\Lambda_{1\infty},\quad \Gamma_f=0.
\end{align}
Then \eqref{sric1} has a constant solution $\Lambda_1(t)\equiv \Lambda_{1\infty}$, and  \eqref{sric2} becomes
\begin{align}\label{richa}
&\dot{\Lambda}_2 =2\hat a  \Lambda_2 +\Lambda_2^2 +\hat Q,\quad
 \Lambda_2(T)=0,
\end{align}
where
\begin{align}\label{haq}
\hat a=\sqrt{A^2+Q}-\frac{G}{2},\quad \hat{Q}=Q\Gamma-\big(A+\sqrt{A^2+Q}\big)G.
\end{align}
To solve \eqref{richa}, let
$\Lambda_2=-\frac{u'}{u}$.
Then \eqref{richa} leads to
\begin{align}
u''-2\hat au'+\hat{Q}u=0. \label{odeu}
\end{align}
Denote
$$
\hat\Delta = \hat a^2-\hat Q=\frac{1}{4}(2A+G)^2+Q(1-\Gamma).
$$

\begin{proposition} \label{prop:global}
The Riccati ODE
\eqref{richa} has a unique solution on $[0,T]$ for all $T>0$
 under either  of the two conditions:
i) $\hat Q\le 0$;
ii) $0<\hat Q\le \hat a^2$ and $\hat a>0$.
\end{proposition}

\begin{proof} See Appendix D. \end{proof}

\begin{proposition}
\label{prop:max}
i) If $0<\hat Q\le \hat a^2$ and $\hat a<0$,  the solution of \eqref{richa} is given by
$$
\Lambda_2(t)=
\begin{cases}
\frac{\hat Q\left(e^{\alpha(T-t)}-e^{-\alpha(T-t)}\right)}{\hat\lambda_2 e^{-\alpha(T-t)}-\hat\lambda_1 e^{\alpha(T-t)}}, &\quad
 \mbox{if}\ \hat Q<\hat a^2,\\
 \frac{\hat a^2(T-t)}{\hat a(t-T)-1}, &\quad \mbox{if}\ \hat Q=\hat a^2,
\end{cases}
$$ where  $\alpha= \sqrt{\hat \Delta}$ and
$\hat\lambda_1=\hat a+ \alpha$,  $\hat\lambda_2=\hat a-\alpha$ are solutions to
 the characteristic equation of \eqref{odeu}.

ii) If $\hat Q> \hat a^2$, then
\begin{align}\label{la2}
\Lambda_2(t) =& \frac{\sqrt{\hat{Q}}  \sin\beta(t-T)}{\sin\big(\beta(T-t)+\theta\big)},
\end{align}
where $\beta =\sqrt{\hat Q-\hat a^2}>0$ and
$\hat a+\beta i= \sqrt{\hat Q} e^{i\theta}$ for $\theta\in (0,\pi)$.
\end{proposition}

\begin{proof} See Appendix D.  \end{proof}

\begin{remark}
The assumptions in the four cases in Propositions \ref{prop:global} and \ref{prop:max} are categorized according to the distribution of the two
eigenvalues of the characteristic equation of \eqref{odeu}.
\end{remark}

\begin{remark}
Depending on the value of $T$, the solutions in both i) and ii) of  Proposition \ref{prop:max} may have a maximal existence internal as a proper subset of $[0,T]$.
\end{remark}

\begin{remark}\label{remark:checkT}
If
\begin{align}
  0<\hat Q<\hat a^2, \quad \hat a<0, \label{dc}
  \end{align}
 by Proposition \ref{prop:max},  $\Lambda_2$ has a finite escape time $\hat t \in[0,T)$
satisfying $T-\hat t= \check T := \frac{\ln (\hat\lambda_2/\hat\lambda_1)}{2\alpha}$ if $\check T\in (0, T]$.
\end{remark}

\begin{example}\label{ex1}
Consider the system with
\begin{align*}
 \quad A= -\frac{1}{4}, \quad G=\frac{4}{5}, \quad Q=\frac{1}{16},\quad  \Gamma =\frac{4}{3}.
\end{align*}
It can be verified that the system satisfies \eqref{dc}.
\end{example}

The  parameters in Example \ref{ex1} are constructed by first fixing $A$ and $Q$, and next searching for $(G, \Gamma)$
subject to the two constraints in \eqref{dc}.

\subsection{The TPBV Problem and Non-uniqueness}
\label{sec:sub:no}

For the scalar case  $n=1$, we take
 $M=1$ and $Q_f=\Lambda_{1\infty}$
 so that $\Lambda_1=\Lambda_{1\infty}$.
 Then \eqref{mbbA}
 reduces to the form
\begin{align}
{\mathbb A}_\infty=
\begin{bmatrix}
G-\sqrt{A^2+Q} &-1 \\
Q\Gamma - (A + \sqrt{A^2+Q})G  & \sqrt{A^2+Q}
\end{bmatrix} \in \mathbb{R}^{2\times 2},
\end{align}
which has the characteristic polynomial
$$
|\lambda I-{\mathbb A}_\infty|= \lambda^2-G\lambda +Q\Gamma -(A^2+Q+AG).
$$
Note that for the TPBV problem \eqref{mfgfp} in the fixed point approach to have multiple solutions, a necessary condition is that asymptotic solvability fails by Theorem \ref{theorem:astoda}.
For constructing non-uniqueness results, below we largely impose conditions in Proposition \ref{prop:max} i).  
If
\begin{align}
  \hat \Delta >0,   \label{dege0}
\end{align}
 $|\lambda I-{\mathbb A}_\infty|=0$ has  the real-valued solutions
\begin{align}
\lambda_1= \frac{G}{2}+\sqrt{\hat\Delta},
\quad \lambda_2= \frac{G}{2}-\sqrt{\hat\Delta}. \label{dl12}
\end{align}
Restricting our attention to two distinct real roots will streamline  the
presentation in constructing non-uniqueness examples.
Under \eqref{dege0},
denote
\begin{align*}
&c_1=-\hat a-\sqrt{\hat\Delta},
\quad c_2=- \hat a+\sqrt{\hat\Delta}  .
\end{align*}
 ${\mathbb A}_\infty$   has two eigenvectors
\begin{align*}
v_k=
[1, c_k]^T, \qquad k=1,2.
\end{align*}
corresponding to the eigenvalues  $\lambda_1$ and $\lambda_2$,
respectively. Now for \eqref{phiAode},
we have
$$
\Phi(t, \tau)= [v_1, v_2]
\begin{bmatrix}
e^{\lambda_1 (t-\tau)}  & \\
   &e^{\lambda_2 (t-\tau)}
\end{bmatrix}
[v_1, v_2]^{-1}
$$
as a $2\times 2$ matrix function.
We further calculate
\begin{align}
&\Phi_{21} (t,\tau)
 =\frac{ c_1c_2 ( e^{\lambda_1 (t-\tau)}
- e^{\lambda_2 (t-\tau)})}{2\sqrt{\hat\Delta}},\label{p21}  \\
&\Phi_{22}(t, \tau) = \frac{c_2 e^{\lambda_2 (t-\tau)} - c_1 e^{\lambda_1 (t-\tau)}   }{2\sqrt{\hat\Delta}}. \nonumber
\end{align}

Given the parameters in  \eqref{QGf}, \eqref{mfgfp} becomes
 \begin{align}\label{mfgfp1}
\begin{cases}
\frac{ d\overline X}{dt} = (A - \Lambda_{1\infty}+G)\overline X
-   s ,   \\
 \dot{s}= -(A^T-\Lambda_{1\infty} )s -\Lambda_{1\infty} G\overline X  + {Q} (\Gamma \overline X +\eta ),
\end{cases}
\end{align}
where $\overline X(0)=x_0$ and $s(T)=- \Lambda_{1\infty}\eta_f$.

In order to construct models with non-uniqueness results, here we
 treat $T$ and $x_0$ in \eqref{mfgfp1} as  adjustable parameters.

\begin{proposition} \label{prop:that}
Assume $Q_f=\Lambda_{1\infty}$.
 If \eqref{dc} holds,
then $\Phi_{21}(T,0)\ne 0$ for all $T>0$ and
there exists a unique  $\hat T>0$ such that $\Phi_{22}(\hat T,0)=0$.
\end{proposition}

\begin{proof}
It can be shown that \eqref{dc}  holds if and only if
\begin{align} \label{dc2}
\hat\Delta>0, \quad c_1>0,
\end{align}
which implies that   $$0<\lambda_2<\lambda_1, \quad 0<c_1<c_2,
$$
where $\lambda_1$ and $\lambda_2$ are given by \eqref{dl12}.
It is clear that $\Phi_{21}(T,0)\ne 0$.
Note that  $\Phi_{22}(T,0)=0$ if and only if
$c_2 e^{\lambda_2 T} = c_1 e^{\lambda_1 T}, $
for which we uniquely solve
$
T= \hat T  :=\frac{\ln (c_2/c_1) }{2\sqrt{\hat\Delta}} = \check T>0$; see Remark \ref{remark:checkT} for $\check T$.
 \end{proof}

For constructing the  TPBV problem below, we suppose
 the assumptions of Proposition \ref{prop:that} hold, and uniquely solve $\hat x_0$ from
\begin{align}\label{hatx0}
&\Phi_{21}(\hat T, 0) \hat x_0 +
\int_0^{\hat T} \Phi_{22}(\hat T, \tau)Q\eta d\tau
= -\Lambda_{1\infty}  \eta_f
\end{align}
since $\Phi_{21}(\hat T, 0)\ne 0$.
We calculate
\begin{align}
\int_0^{\hat T}\Phi_{22}(\hat T, \tau) d\tau   
 =& \frac{c_2\lambda_1 (e^{\lambda_2 \hat T}-1) -c_1\lambda_2 (e^{\lambda_1 \hat T}-1)}{2\sqrt{\hat \Delta}\lambda_1\lambda_2}
. \nonumber
\end{align}

Now for the scalar case with $M=1$, $Q_f=\Lambda_{1\infty}$, $\Gamma_f=0$ and $T=\hat T$,  \eqref{mfgfp1}
specializes to the  TPBV problem
\begin{align}\label{mfgfp0}
\begin{cases}
\frac{ d\overline X}{dt} = (A - \Lambda_{1\infty}+G)\overline X -    s ,   \\
 \dot{s}(t)= -(A^T-\Lambda_{1\infty} )s -\Lambda_{1\infty} G\overline X  + {Q} (\Gamma \overline X +\eta ),
\end{cases}
\end{align}
where $\overline X(0)=\hat x_0$ and $s(\hat T)=- \Lambda_{1\infty}\eta_f$.

\begin{proposition} Assume \eqref{QGf} and \eqref{dc} hold. Then  a solution $(\overline X, s)$ of  \eqref{mfgfp0} can be obtained by taking any initial condition $s(0)$. Therefore, \eqref{mfgfp0} has an infinite number of solutions.
\end{proposition}

\begin{proof} Recalling \eqref{sTs0},
  \eqref{mfgfp0} is solvable if and only if
one can find  $s(0)$ to satisfy
\begin{align}
&\Phi_{21}( \hat T, 0)\hat x_0 +\Phi_{22}( \hat T, 0)s(0)+
\int_0^{ \hat T} \Phi_{22}( \hat T, \tau)Q\eta  d\tau \nonumber \\
=& -\Lambda_{1\infty}  \eta_f. \label{phixs}
\end{align}
By  \eqref{hatx0} and $\Phi_{22}(\hat T, 0)=0$, \eqref{phixs}
 holds for any choice of $s(0)$.~\end{proof}

\begin{figure}[t]
\begin{center}
\psfig{file=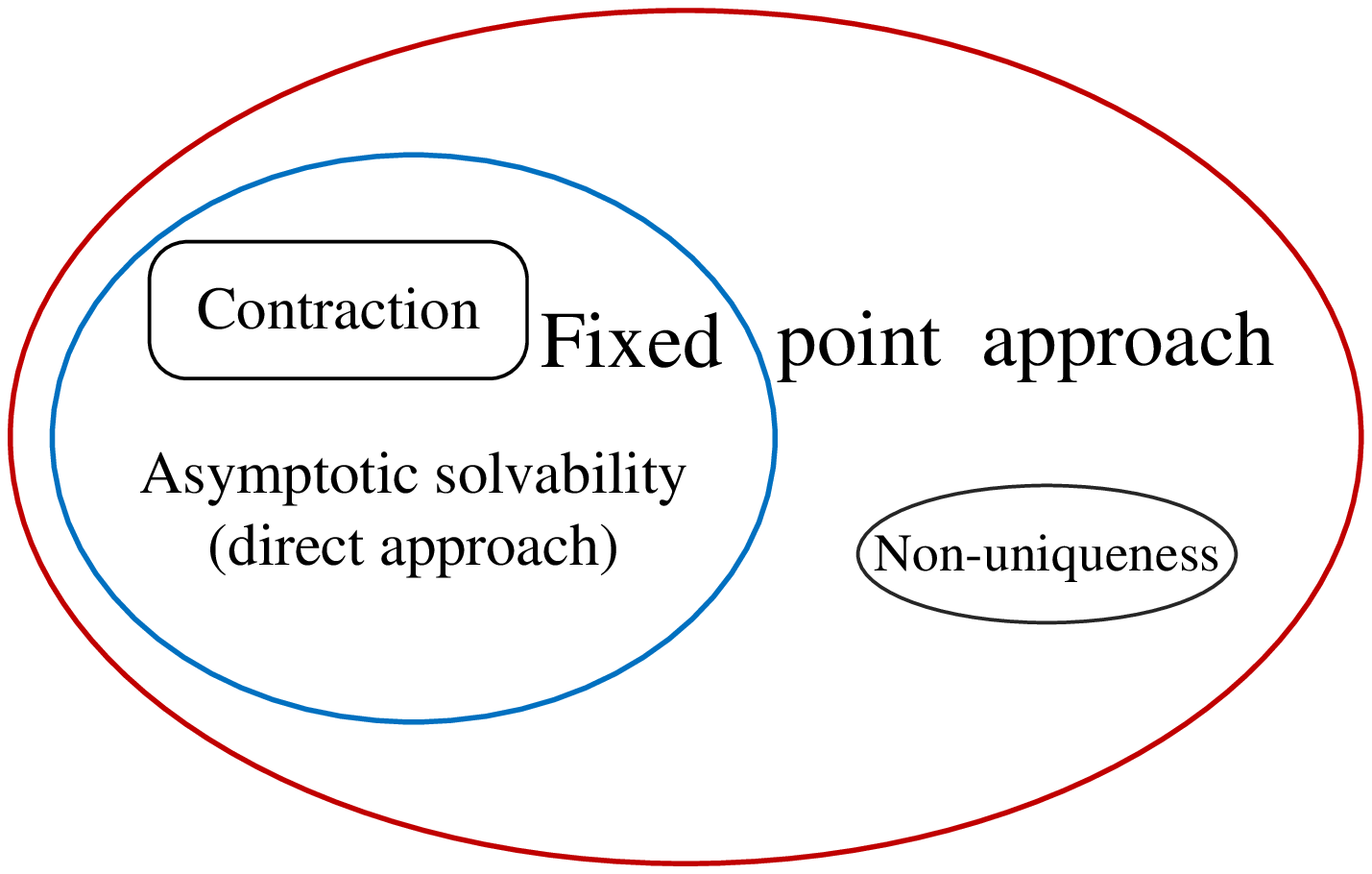, width=2.1in, height=1.4in}
\end{center}
\caption{Comparison of the two  approaches  }
 \label{fig:diadomain}
\end{figure}

\subsection{Comparison of Two Approaches}
Consider the system given by Example \ref{ex1} with time horizon $[0,T]$.
It satisfies \eqref{dc}. Then $\check T=\hat T$.

If we take $T\in (0, \hat T)$, then asymptotic solvability holds and
the TPBV problem \eqref{mfgfp1} has a unique solution by Theorem
\ref{theorem:astoda}.

 If $T=\hat T$, then $\Lambda_2$ has a finite escape time at $t=0$ implying no asymptotic solvability. However, in this case the TPBV problem \eqref{mfgfp0} has an infinite number of solutions, which in turn can be used to construct an infinite number of $\epsilon$-Nash equilibria for the $N$-player game.

If $T>\hat T$, asymptotic solvability fails but  \eqref{mfgfp1} has a unique solution since $\Phi_{22}(T,0)\ne 0$ by Proposition \ref{prop:that}.

Based on Theorems \ref{theorem:astoda} and \ref{theorem:conAS}, and the comparison above, the relation between the two approaches is illustrated in
Fig. \ref{fig:diadomain}. The rectangle region represents models satisfying the contraction condition $\kappa_0<1$ in Theorem \ref{theorem:conAS}.

\section{Long Time Behavior}

\label{sec:long}

For this section, we make the following assumption:

(H1) The pair $(A,B)$ is stabilizable,
and the pair $(A, Q^{\frac12})$ is
detectable.

Within the setup of continuous time dynamical systems, a matrix $Z\in \mathbb{R}^{k\times k}$ is called stable or Hurwitz if all its eigenvalues have a strictly negative real part.

\subsection{Steady State Form of  Riccati ODEs}

For \eqref{d11}, we introduce the ARE
$$
\Lambda_{1\infty} M \Lambda_{1\infty} -(\Lambda_{1\infty} A +A^T \Lambda_{1\infty})-Q=0.
$$
Note that under (H1)
 there exists a unique solution  $\Lambda_{1\infty}\ge 0$
 from the class of positive semi-definite matrices.
Corresponding to   \eqref{d21}, we introduce the  algebraic equation
\begin{align}
0 =\ & \Lambda_{1\infty} M \Lambda_{2\infty}+ \Lambda_{2\infty} M
\Lambda_{1\infty}+ \Lambda_{2\infty}M\Lambda_{2\infty} \nonumber\\
                    & - (\Lambda_{1\infty} G + \Lambda_{2\infty} (A+G) +A^T\Lambda_{2\infty})
                     +Q\Gamma,\label{ared21}
\end{align}
which is a non-symmetric algebraic Riccati equation (NARE).
When \eqref{ared21} has a solution in $\mathbb{R}^{n\times n}$,
it is possible that multiple such solutions exist.
The question is how to determine a solution of interest, and this amounts to imposing appropriate constraints on the solution. For related methods on choosing a desirable solution of NAREs by fulfilling some stability conditions, see \cite{KS02}.

\subsection{Stabilizing Solution}
Suppose $\Lambda_{2\infty} \in \mathbb{R}^{n\times n} $ is a solution to \eqref{ared21}.
Denote
\begin{align}
&A_G= A-M (\Lambda_{1\infty}+\Lambda_{2\infty})+G, \label{ag}\\
& A_{M}= A -M(\Lambda_{1\infty}+\Lambda_{2\infty}^T).\label{am}
\end{align}
To motivate the restrictions to be imposed on $\Lambda_{2\infty}$,
we examine the two ODEs \eqref{theta_1_new}
and \eqref{clxbar}, where the latter is  the closed-loop mean field dynamics.
We start by checking the stability of the solution of \eqref{theta_1_new}
 when
$t$ is simply allowed to tend to $-\infty$.
If $\Lambda_2(t)$ can converge to a limit $\Lambda_{2\infty}$ at all, it is well justified to study the stability of the limiting ODE
\begin{align}
\dot \chi_1& = [(\Lambda_{1\infty}+\Lambda_{2\infty}) M -A^T] \chi_1+Q\eta \nonumber\\
    &= -A_M^T \chi_1 +Q\eta, \qquad t\in (-\infty, T], \label{ch1inf}
\end{align}
which is constructed by replacing $(\Lambda_1(t), \Lambda_{2}(t))$ by
$(\Lambda_{1\infty}, \Lambda_{2\infty})$ in
\eqref{theta_1_new}.  The solution of \eqref{ch1inf} converges to a constant vector $\chi_{1\infty}$
 as $t\to -\infty$ if $A_M$ is Hurwitz.
 Thus  the generation of  stable long time behavior suggests we impose a stability condition on $A_M$.
 For \eqref{clxbar} we  similarly introduce a limiting ODE of the form
 \begin{align}
 \frac{d\bar X}{dt} &= \left(A-M (\Lambda_{1\infty} +\Lambda_{2\infty})+G\right)\bar X -M \chi_{1\infty} \nonumber \\
&= A_G \bar X   -M \chi_{1\infty},\qquad t\in [0,\infty),
 \end{align}
and further introduce a stability condition on $A_G$ in order to have a stable solution.

\begin{definition}
$\Lambda_{2\infty}\in \mathbb{R}^{n\times n}$ is called a stabilizing solution
of \eqref{ared21}
if it satisfies \eqref{ared21} and both
 $A_G$ and $A_M$
are Hurwitz.
\end{definition}

If $\Lambda_{2\infty}$ is a stabilizing solution,  it has the interpretation as a locally stable equilibrium point of the Riccati ODE \eqref{d21}.  We take a limiting form of \eqref{d21} by replacing $\Lambda_1$ by $\Lambda_{1\infty}$ and for convenience of analysis  next  reverse time to obtain the new equation
\begin{align}
\dot Y(t) = & -\Lambda_{1\infty} M Y- Y M\Lambda_{1\infty}- YMY\nonumber\\
                    & + (\Lambda_{1\infty} G + Y (A+G)
                    +A^TY)  -Q\Gamma, \quad t\ge 0, \label{Yode}
\end{align}
for which we take a general initial condition $Y(0)$.
The linearized ODE for
\eqref{Yode} around $\Lambda_{2\infty}$
 is \begin{align}
\dot Z(t)=
A_M^T Z+ ZA_G , \qquad t\ge 0,  \nonumber
\end{align}
which is a Sylvester ODE with solution
$
Z(t)=  e^{A_M^T t} Z(0) e^{A_Gt}.
$
This ODE is asymptotically stable given any $Z(0)$
if the matrices $A_G$ and $A_M$ are Hurwitz.

We proceed to determine conditions for existence of a stabilizing solution.
Denote
\begin{align}
{\mathbb A}_\infty=\begin{bmatrix}
A- M \Lambda_{1\infty} +G   & -M\\
Q\Gamma -\Lambda_{1\infty} G & -A^T+ \Lambda_{1\infty} M
\end{bmatrix}\in \mathbb{R}^{2n\times 2n},
\end{align}
which may be viewed as a steady state form of
${\mathbb A}(t)$ in \eqref{mbbA}.

Let ${\mathbb A}_o\in \mathbb{R}^{k\times k}$ be any matrix.
An $l$-dimensional subspace ${\cal V}$ of $\mathbb{R}^{k}$ is called
an invariant subspace of  ${\mathbb A}_o$ if
${\mathbb A}_o {\cal V} \subset {\cal V}$; in this case
${\mathbb A}_o V =V A_o$ for some $A_o \in\mathbb{R}^{l\times l}$
where $V\in \mathbb{R}^{k\times l}$ and ${\rm span} \{V\}={\cal V}$.
If $A_o$ is Hurwitz, ${\cal V}$ is called a stable invariant subspace.
Below we give some standard definitions related to structural properties of an invariant subspace (see e.g. \cite{BIM12,LR95}).
For $1\le l<k$, an $l$-dimensional  invariant subspace ${\cal V}_g$ of ${\mathbb A}_o\in \mathbb{R}^{k\times k}$  is called a graph subspace if ${\cal V}_g$ is spanned by the columns of a $k\times l$ matrix whose leading $l\times l$ submatrix (i.e., its first $l$ rows) is invertible. The $k $ eigenvalues of ${\mathbb A}_o$  have a strong $(k_1,k_2)$ c-splitting if the open left half plane and the open right half plane  contain $k_1$ and $k_2$ eigenvalues, respectively, for $k_1\ge 1$, $k_2\ge 1$, $k_1+k_2=k$.

We introduce the following condition on
$\mathbb{A}_\infty$:

(H$_{g}$)  The eigenvalues of
${\mathbb A}_\infty$ are strong $(n,n)$ $c$-splitting and the
associated $n$-dimensional stable
invariant subspace is a graph subspace.

\begin{theorem} \label{theorem:stab}
i) The NARE \eqref{ared21} has a stabilizing solution $\Lambda_{2\infty}$ if and only if (H$_{g}$) holds.

ii) If (H$_g$) holds, \eqref{ared21} has  a unique stabilizing solution.  \end{theorem}

\begin{proof} i) Step 1. To show necessity, suppose that $\Lambda_{2\infty}$ is a stabilizing solution.
Denote
$$
K=\begin{bmatrix}
I_n & 0 \\
\Lambda_{2\infty} &I_n
\end{bmatrix}.$$
Since \eqref{ared21} holds, it can be checked that
\begin{align} \label{KAK}
K^{-1} {\mathbb A}_\infty K=
\begin{bmatrix}
A_G  & -M \\
0 &-A_M^T
\end{bmatrix}.
\end{align}
By the definition of a stabilizing solution, $A_G$ and $A_M$ are Hurwitz. So $-A_M^T$ has all its eigenvalues in the open right half plane.
Therefore, the eigenvalues of  ${\mathbb A}_\infty$ have a strong $(n,n)$ $c$-splitting.
 Now the columns of
  \begin{align}
\begin{bmatrix}
I_n\\
\Lambda_{2\infty}
\end{bmatrix} \nonumber
   \end{align}
span the $n$-dimensional stable invariant subspace of ${\mathbb A}_\infty$ as a graph subspace.

Step 2. We continue to  show sufficiency. Suppose
the columns of the matrix
\begin{align} \label{U12}
\begin{bmatrix}
U_1\\
U_2
\end{bmatrix}\in \mathbb{R}^{2n\times n}
\end{align}
spans the $n$-dimensional stable invariant subspace of $\mathbb{A}_\infty$, where
$U_1$ is invertible.
We take
\begin{align}\label{u12}
\Lambda_{2\infty}= U_2U_1^{-1}.
\end{align}
Then one can
directly verify that
$\Lambda_{2\infty}$ solves \eqref{ared21}
(see e.g. \cite[Corollary 2.2,  pp. 34]{BIM12}), and \eqref{KAK} holds where $A_M$ and $A_G$ in \eqref{ag}--\eqref{am}  are determined by
use of \eqref{u12}.
Since $A_G$ is associated with the stable invariant subspace, it is necessarily a Hurwitz matrix. Since the eigenvalues of ${\mathbb A}_\infty$ are $(n,n)$ $c$-splitting, $-A_M^T$ has $n$ eigenvalues
in the open right half plane, which implies that $A_M$ is Hurwitz. Hence, \eqref{ared21} has  a
stabilizing solution.

ii) Suppose $\Lambda_{2\infty}$ and $\bar \Lambda_{2\infty}$ are two stabilizing solutions. Denote
$$
Y= \begin{bmatrix}
I_n\\
\Lambda_{2\infty}
\end{bmatrix}, \qquad \bar Y= \begin{bmatrix}
I_n\\
\bar \Lambda_{2\infty}
\end{bmatrix}.
$$
By Step 1,  $\mbox{span}\{Y\}= \mbox{span}\{\bar Y\}$ since they both are equal to the $n$-dimensional stable invariant subspace of
${\mathbb A}_\infty$. Now for each $1\le i\le n$, the $i$th column $Y_i$ of $Y$ is in $\mbox{span}\{\bar Y\}$, which further implies that $Y_i$ is equal to the $i$th column of $\bar Y$. Therefore $\Lambda_{2\infty}=\bar \Lambda_{2\infty}$, and uniqueness follows.
\end{proof}

Theorem \ref{theorem:stab} presents a qualitative criterion on the
existence of a stabilizing solution to the NARE \eqref{ared21}.
Step 2 in the proof further provides a computational procedure.
When (H$_{g}$) holds, one may choose any $n$ basis vectors of the
$n$-dimensional stable invariant subspace to form the matrix in \eqref{U12} and the resulting matrix $U_1\in \mathbb{R}^{n\times n}$
is necessarily invertible. Subsequently one uses \eqref{u12} to
find the stabilizing solution. In fact, there is a simple  means to test whether (H$_{g}$) holds. If the eigenvalues of $\mathbb{A}_\infty$ are strong $(n,n)$ c-splitting, one takes any $n$ basis vectors of
the stable invariant subspace to form a matrix as in \eqref{U12} with $U_1$ to be further checked.
Finally, if $U_1$ is invertible, (H$_{g}$) holds; and (H$_{g}$)   fails otherwise.

\section{Numerical Examples}

\label{sec:nume}

\subsection{Asymptotic Solvability }
Consider the Riccati ODEs \eqref{d11} and \eqref{d21} with $n=1$.

\begin{example} \label{ex:posiexample}
Take the parameters $A=0.2$, $B=G=Q=R=1$, $Q_f=\Lambda_{1\infty}$, $\Gamma=1.2$, $\Gamma_f=0$. Then
  \eqref{d11} gives $\Lambda_1(t)\equiv \Lambda_{1\infty}$ and \eqref{d21} becomes
\begin{align*}
\dot{\Lambda}_2 &= 2\Lambda_{1\infty}\Lambda_2+\Lambda_2^2-(\Lambda_{1\infty}+1.4\Lambda_2)+1.2, \quad \Lambda_2(T)=0.
\end{align*}
By verifying condition i) in Proposition \ref{prop:global}, we see that  $\Lambda_2$ has a solution on  $[0,T]$ for any $T>0$.
So asymptotic solvability holds.
\end{example}

\begin{example} \label{ex:counterexample}
Take
$Q_f=0$ and  $T=3$. All other parameters are the same as in
Example \ref{ex:posiexample}.
Now \eqref{d11} and \eqref{d21} reduce to
\begin{align*}
\dot{\Lambda}_1 &= \Lambda_1^2-0.4\Lambda_1-1, \quad \Lambda_1(T)=0, \\
\dot{\Lambda}_2 &= 2\Lambda_1\Lambda_2+\Lambda_2^2-(\Lambda_1+1.4\Lambda_2)+1.2, \quad \Lambda_2(T)=0.
\end{align*}
$\Lambda_1$ can be solved explicitly on $[0,T]$.
Fig. \ref{fig:bu}  shows that $\Lambda_2$
does not have a solution on the whole interval $[0,T]$ implying no asymptotic solvability.
\end{example}

Examples \ref{ex:posiexample}
and \ref{ex:counterexample} reveal a significant role of $\Lambda_1$ in affecting the existence interval of $\Lambda_2$.
\begin{example}
 Consider a system with parameters in  Example \ref{ex1}
and $T=35$. Following the notation in Proposition \ref{prop:max}, then
\begin{align*}
 &\hat a=-  0.046447, \quad
 \hat Q = 4.906209 \times 10^{-4}.
\end{align*}
So
$$
0<\hat Q < \hat a^2=0.002157 ,
$$ and $\frac{1}{2\alpha} \ln ({\hat\lambda_2}/{\hat\lambda_1})=33.587095$.
By Proposition \ref{prop:max},
 $\Lambda_2(t)$ has a finite escape time at
 $\hat t\approx 1.4129$. The TPBV problem \eqref{mfgfp1} has a unique solution since $\Phi_{22}( 35,0)\ne 0$.
\end{example}

\subsection{Non-uniqueness }

Consider a system with parameters  in Example \ref{ex1}
and $\eta=\eta_f=1$.
Following the notation in subsection \ref{sec:sub:no},
\begin{align*}
&\hat \Delta = 0.001667,    \quad
c_1=0.005622, \quad
 c_2= 0.087271,
\end{align*}
which satisfy the conditions in Proposition \ref{prop:that},
and  further determine
\begin{align*}
&\hat T=33.587095,\quad
 \hat x_0=-0.394732.
\end{align*}
Fig. \ref{fig:nu} displays $\Phi_{21}(T,0)$ and $\Phi_{22}(T,0)$,
where $T$ is treated as a variable. It shows that $\Phi_{22} (T,0)=0$
when $T=\hat T$.

Now consider the model \eqref{stateXi}--\eqref{costJi}
with time horizon $[0,\hat T]$.
In this case, we  have no
asymptotic solvability since $\Lambda_2$ has the maximal existence interval $(0, \hat T]$.
However, the TPBV problem \eqref{mfgfp0} has an infinite number of solutions.
\begin{figure}[t]
\begin{center}
\begin{tabular}{c}
\psfig{file=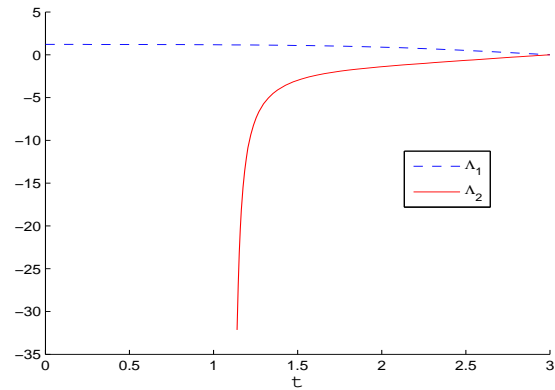, width=3.4in, height=2.2in}
\end{tabular}
\end{center}
\caption{ $\Lambda_2$ has a maximal existence interval small than $[0,T]$ }  \label{fig:bu}
\end{figure}

\begin{figure}[t]
\begin{center}
\begin{tabular}{c}
\psfig{file=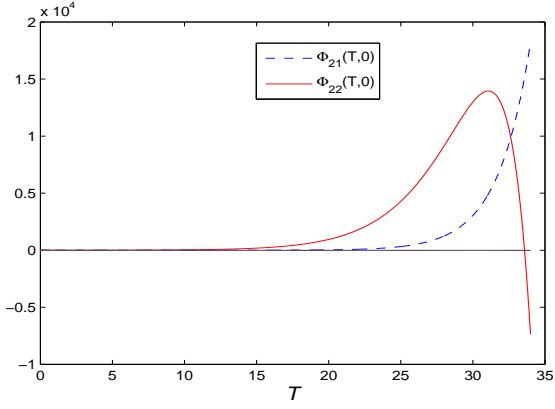, width=3.4in, height=2.2in}
\end{tabular}
\end{center}
\caption{$\Phi_{21}(T,0) $ and $\Phi_{22} (T,0)$ as  a function of $T$. }  \label{fig:nu}
\end{figure}

\subsection{Stabilizing Solution for the NARE \eqref{ared21} }

\begin{example}\label{ex:posivec}
We take
\begin{align}
A=
\begin{bmatrix}
1 & 1\\
-0.5 & 1
\end{bmatrix}, \
B=
\begin{bmatrix}
0\\
 1
\end{bmatrix}, \
\Gamma=
\begin{bmatrix}
0.9 & 0.1\\
0 & 0.9
\end{bmatrix}, \
\eta=
\begin{bmatrix}
1\\
 0
\end{bmatrix}, \nonumber
\end{align}
and $G=Q=I_2$,
$R=1$.
Then \eqref{ared21} has a stabilizing solution
$$
\Lambda_{2\infty} =
\begin{bmatrix}
 16.238985   & 4.099679 \\
   4.132523   & 1.570208
\end{bmatrix}.
$$
In fact, the columns of the matrix
\begin{align}
\begin{bmatrix}
-0.167388 & -0.161703\\
   0.448957 &  0.742511\\
  -0.877636  &  0.418170 \\
   0.013220  &  0.497657
\end{bmatrix} \nonumber
\end{align}
span the stable invariant subspace of $\mathbb{A}_\infty$ as a graph subspace. $\mathbb{A}_\infty$ has
the eigenvalues
\begin{align*}
-1.022350 \pm 0.730733i, \quad
   2.022350 \pm 0.707903i.
\end{align*}
\end{example}

\begin{example}
We take $G=-1.2 I_2$ and all other parameters are the same as
 in Example \ref{ex:posivec}. Then there exists no stabilizing solution
 $\Lambda_{2\infty}$ since in this case ${\mathbb A}_\infty$ has the eigenvalues
\begin{align*}
-1.090328 \pm 0.762501i,\quad
  -0.109672 \pm 0.692413i.
\end{align*}
\end{example}

\section{Conclusion}
\label{sec:con}

This paper investigates an asymptotic solvability problem in LQ
mean field games, and studies its connection with  the
 fixed point approach which involves a TPBV problem.
For  asymptotic solvability
we derive a necessary and sufficient condition via a non-symmetric Riccati ODE. It is shown that asymptotic solvability provides a sufficient condition
for the TPBV problem in the fixed point approach to have a unique solution.
We identify situations for the TPBV problem to be solvable
or have multiple solutions when asymptotic solvability
does not hold. The long time behavior of the non-symmetric Riccati ODE in the asymptotic solvability problem is addressed by studying the stabilizing solution to a non-symmetric algebraic Riccati equation.

The re-scaling technique used in studying asymptotic solvability can be extended to more general models in terms of dynamics, interaction and information patterns \cite{CK17,HWW16,H10,HN11}. This will be reported in  our future work.

\section*{Appendix A: Proof of Theorem \ref{theorem:Prep3}}
\renewcommand{\theequation}{A.\arabic{equation}}
\setcounter{equation}{0}
\renewcommand{\thetheorem}{A.\arabic{theorem}}
\setcounter{theorem}{0}

\begin{lemma}
\label{lem:Prep4}
 We assume that \eqref{DE3_P} has  a solution
 $(\mbP_1(t), \cdots,
\mbP_N(t))$  on $[0,T]$. Then the following holds.

i) ${\mbP}_1(t)$  has the  representation
\begin{align}\label{P_matrix}
{\mbP}_1(t)=\begin{bmatrix}
\Pi_1(t) &\Pi_2(t) &\Pi_2(t)&\cdots &\Pi_2(t) \\
\Pi_2^T(t) &\Pi_3(t) &\Pi_4(t)&\cdots &\Pi_4(t)\\
\Pi_2^T(t) &\Pi_4(t)&\Pi_3(t) &\cdots &\Pi_4(t)\\
\vdots          & \vdots        &  \vdots        &\ddots &\vdots \\
\Pi_2^T(t) &\Pi_4(t) &\Pi_4(t)&\cdots &\Pi_3(t)
\end{bmatrix},
\end{align}
where $\Pi_1$, $\Pi_3$ and $\Pi_4$ are $n\times n$ symmetric matrices.

ii) For $i>1$,  $\mbP_i(t)= J_{1i}^T \mbP_1(t) J_{1i}$.
\end{lemma}

\begin{proof} Step 1. It is straightforward to show
\begin{align}
J_{23}^T  \Psi J_{23}= \Psi, \nonumber
\end{align}
where $\Psi$ is any matrix from $\widehat \mbA$, $\mbQ_i$, $\mbQ_{if}$ and  $\mbB_iR^{-1} \mbB_i^T$, $i\ne 2,3$.
 And moreover,
  \begin{align*}
& J_{23}^T \mbQ_2  J_{23} = \mbQ_3,  \qquad
 J_{23}^T \mbQ_3  J_{23} = \mbQ_2  \\
& J_{23}^T \mbQ_{2f}  J_{23} = \mbQ_{3f}   \qquad
 J_{23}^T  \mbQ_{3f} J_{23} = \mbQ_{2f}  \\
&J_{23}^T \mbB_2R^{-1} \mbB_2^T J_{23} = \mbB_3R^{-1} \mbB_3^T,\\
&J_{23}^T \mbB_3R^{-1} \mbB_3^T J_{23} = \mbB_2R^{-1} \mbB_2^T.
\end{align*}

Denote
$J_{23}^T \mbP_i J_{23}=\mbP_i^\dag$ for $1\le i\le N$.
 Multiplying both sides of \eqref{DE3_P} from the left
by $J_{23}^T$ and next from the right by $J_{23}$,
 we obtain
\begin{align*}
\dot \mbP_i^\dag  = & - \mbP_i^\dag J_{23}^T \widehat \mbA J_{23}
-J_{23}^T \widehat \mbA^T J_{23} \mbP_i^\dag \\
& + \mbP_i^\dag \sum_{k=1}^N J_{23}^T \mbB_k R^{-1} \mbB_k^T J_{23}
\mbP_k^\dag   \\
 &
 + \sum_{k=1}^N \mbP_k^\dag  J_{23}^T \mbB_k R^{-1} \mbB_k^T J_{23}
\mbP_i^\dag \\
&-  \mbP_i^\dag  J_{23}^T \mbB_i R^{-1} \mbB_i^T J_{23} \mbP_i^\dag - J_{23}^T \mbQ_i  J_{23}.
\end{align*}
Hence for $i\ne 2, 3$,
\begin{align*}
\dot \mbP_i^\dag  = & - \mbP_i^\dag  \widehat \mbA
- \widehat \mbA^T  \mbP_i^\dag
 + \mbP_i^\dag \sum_{k\notin \{2,3\}}^N  \mbB_k R^{-1} \mbB_k^T
\mbP_k^\dag    \\
&+\mbP_i^\dag ( \mbB_3 R^{-1} \mbB_3^T
\mbP_2^\dag+ \mbB_2 R^{-1} \mbB_2^T
\mbP_3^\dag  ) \\
&+ \sum_{k\notin\{2,3\}}^N \mbP_k^\dag   \mbB_k R^{-1} \mbB_k^T
\mbP_i^\dag \\
&+(\mbP_2^\dag   \mbB_3 R^{-1} \mbB_3^T
 +\mbP_3^\dag   \mbB_2 R^{-1} \mbB_2^T )
\mbP_i^\dag   \\
&-  \mbP_i^\dag  \mbB_i R^{-1} \mbB_i^T\mbP_i^\dag  -  \mbQ_i,
\end{align*}
where $\mbP_i^\dag (T)=\mbQ_{if}$. Similarly,
we can write the equations for
$\dot \mbP_2^\dag$ and $\dot \mbP_3^\dag$ for which we omit the details.  Note that $\mbP_2^\dag(T) = \mbQ_{3f}$ and $\mbP_3^\dag(T) = \mbQ_{2f}$. Subsequently, we list the
$N$ equations by the order of
$\dot\mbP_1^\dag,\ \dot\mbP_3^\dag,\ \dot\mbP_2^\dag ,\ \dot\mbP_4^\dag
, \ \cdots, \ \dot\mbP_N^\dag$, and it turns out that
$$(J_{23}^T \mbP_1 J_{23}, \ J_{23}^T \mbP_3 J_{23},\  J_{23}^T \mbP_2 J_{23},\ J_{23}^T \mbP_4 J_{23}, \cdots, J_{23}^T \mbP_N J_{23})$$
satisfies \eqref{DE3_P}  as $(\mbP_1(t), \cdots,
\mbP_N(t))$ does.

Step 2.  For $1\le i\le N$,  denote
$\mbP_i= (\mbP_i^{jk})_{1\le j,k\le N}, $
where each $\mbP_i^{jk}$ is an $n\times n$ matrix. By Step 1,
$
\mbP_1= J_{23}^T \mbP_1 J_{23},
$
which implies
\begin{align}
\mbP_1^{12}=\mbP_1^{13},\quad  \mbP_1^{22}=P_1^{33},
\quad \mbP_1^{23}=\mbP_1^{32}.\label{P123}
\end{align}
Repeating the above procedure by using $J_{2k}$, $k\ge 4$, in place of $J_{23}$, we obtain
$$
\mbP^{12}_1=\mbP_1^{13}=\cdots=\mbP_1^{1N}, \qquad \mbP_1^{22}=\mbP_1^{33}=\cdots =\mbP_1^{NN}.
$$
We similarly obtain $\mbP_1= J_{34}^T \mbP_1 J_{34}$, and this gives
$$
\mbP^{23}_1=\mbP^{24}_1.
$$
Repeating a similar argument,
we can check all other remaining off-diagonal submatrices. Since $\mbP_1$ is symmetric (also see Remark \ref{remark:P}), $(\mbP_1^{23})^T= \mbP_1^{32}$,
which implies that $\mbP_1^{23}$  is symmetric by \eqref{P123}. By the above method we can show that the off-diagonal submatrices $P_1^{ij}$, where $i\ne j$ and $2\le i\le N$, $2\le j\le N$, are equal and symmetric.
Therefore we obtain the representation of $\mbP_1$.

Step 3.  We can verify that
$$(J_{12}^T \mbP_2 J_{12}, \ J_{12}^T \mbP_1 J_{12},\  J_{12}^T \mbP_3 J_{12}, \cdots, J_{12}^T \mbP_N J_{12})$$ satisfies \eqref{DE3_P} as
$(\mbP_1(t), \cdots,
\mbP_N(t))$ does. Hence $\mbP_2=J_{12}^T \mbP_1 J_{12}$.
All other cases can be similarly checked.
\end{proof}

{\it Proof of Theorem \ref{theorem:Prep3}:}

By Lemma \ref{lem:Prep4}, we have
\begin{align}
&\dot{\Pi_1}(t)  =   \Pi_1 M\Pi_1
                          +(N-1)(\Pi_2M\Pi_2
                          +\Pi_2^T M\Pi_2^T) \nonumber \\
                          &\quad\qquad - \Big(\Pi_1 (A+\frac{G}{N})
                          +(A^T+\frac{G^T}{N})\Pi_1 \Big)\nonumber \\
                          &\quad\qquad-(1-\frac{1}{N}) (\Pi_2G
                          +G^T\Pi_2^T) \nonumber \\
                          &\quad\qquad -({I}-\frac{\Gamma^T}{N})Q({I}-\frac{\Gamma}{N}),
                        \label{d1}   \\
                  &\Pi_1(T)=({I}-\frac{\Gamma_f^T}{N})
                 Q_f({I}-\frac{\Gamma_f}{N}),\nonumber
              \end{align}
and
\begin{align}
&\dot{\Pi_2}(t) =  \Pi_1 M\Pi_2+\Pi_2M\Pi_1
+\Pi_2^TM\Pi_3\nonumber \\
             &\quad\qquad+(N-2)\Pi_2M\Pi_2
           +(N-2)\Pi_2^TM\Pi_4 \nonumber \\
             &\quad\qquad-\Big(\Pi_1 \frac{G}{N}+\frac{G^T}{N}\Pi_3
          +\frac{N-2}{N}G^T\Pi_4 \nonumber \\
       &\quad\qquad+\Pi_2(A+\frac{N-1}{N}G)
      +(A^T+\frac{G^T}{N})\Pi_2\Big) \nonumber \\
& \quad\qquad+ ({I}-\frac{\Gamma^T}{N})Q\frac{\Gamma}{N},\label{d2} \\
&\Pi_2(T)=-({I}-\frac{\Gamma_f^T}{N})Q_f\frac{\Gamma_f}{N},\nonumber
\end{align}
and
\begin{align}
&\dot{\Pi_3}(t) = \Pi_2^TM\Pi_2+\Pi_3M\Pi_1
+\Pi_1 M\Pi_3\nonumber \\
                       &\quad\qquad+(N-2)(\Pi_4M\Pi_2
                       +\Pi_2^TM\Pi_4) \nonumber\\
                       &\quad\qquad - \Big( \frac{1}{N}(\Pi_2^TG
                       +G^T\Pi_2)
                        \nonumber \\
                       &\qquad\qquad+\Pi_3(A+\frac{G}{N})
                       +(A^T+\frac{G^T}{N})\Pi_3\nonumber \\
                       & \qquad\qquad +\frac{N-2}{N}(\Pi_4G +
                       G^T\Pi_4)
                        \Big)
 - \frac{\Gamma^T}{N}Q\frac{\Gamma}{N}, \label{d3} \\
& \Pi_3(T)=\frac{\Gamma_f^T}{N}Q_f\frac{\Gamma_f}{N},\nonumber
\end{align}
and
\begin{align}
&\dot{\Pi_4}(t) =\Pi_2^TM\Pi_2
            +\Pi_4 M\Pi_1
             +\Pi_1 M\Pi_4
                       +\Pi_3M\Pi_2
                        \nonumber \\
                       &\quad\qquad
                      +\Pi_2^TM\Pi_3  +(N-3)(\Pi_4 M\Pi_2
                       +\Pi_2^TM\Pi_4 ) \nonumber \\
                       &\quad\qquad - \Big( \frac{1}{N}(\Pi_2^T G
                       +G^T\Pi_2
                       +\Pi_3G
                       +G^T\Pi_3) \nonumber \\
                       &\quad\qquad+ \Pi_4(A+\frac{N-2}{N}G)
                       +(A^T+\frac{N-2}{N}G^T)\Pi_4 \Big)\nonumber \\
&\quad\qquad - \frac{\Gamma^T}{N}Q\frac{\Gamma}{N}, \label{d4} \\
& \Pi_4(T)=\frac{\Gamma_f^T}{N}Q_f\frac{\Gamma_f}{N}.\nonumber
\end{align}

The last two ODEs lead to
\begin{align*}
\frac{d}{dt}{(\Pi_3-\Pi_4)}& = (\Pi_3-\Pi_4)(M\Pi_1-M\Pi_2-A)\\
&+ (\Pi_1 M -\Pi_2^T M -A^T) (\Pi_3-\Pi_4),
\end{align*}
where $\Pi_3(T)-\Pi_4(T)=0$. This can be viewed as a linear ODE once $\Pi_1$ and $\Pi_2$ are fixed.  Therefore $\Pi_3\equiv \Pi_4$ on $[0,T]$.
This completes the proof. \endproof

\section*{Appendix B: Proof of Theorem \ref{theorem:depen}}
\renewcommand{\theequation}{B.\arabic{equation}}
\setcounter{equation}{0}
\renewcommand{\thetheorem}{B.\arabic{theorem}}
\setcounter{theorem}{0}

\begin{proof}
i)
 We can find a constant $C_z$ such that $\sup_{0\leq t\leq T} |x^z(t)|\leq C_z$, and $\sup_{0<\epsilon\le 1}|z_\epsilon|\le C_z$.
Fix the open ball $B_{2C_z}(0)$. For $x, y\in B_{2C_z}(0)$ and $t\in [0,T]$, we have
\begin{align*}
|\phi(t, x)-\phi(t, y)|\leq \mbox{Lip} (2C_z) |x-y|.
\end{align*}

For each $\epsilon\le 1$,
by (A1)--(A3), \eqref{dot_y} has a solution $y^{\epsilon}(t)$  defined either (a) for all $t\in[0, T]$ or (b) on a maximal interval $[0, t_{\max})$ for some $0<t_{\max}\le T$.

Below we  show that for all small $\epsilon$, (b) does not occur.
We prove by contradiction. Suppose
for any small $\epsilon_0>0$, there exists $0<\epsilon<\epsilon_0$ such that (b) occurs with the corresponding
 $0< t_{\max} \le T$.
Since $[0, t_{\max})$ is the maximal existence interval,
we have $\lim_{t\uparrow t_{\max}}|y^\epsilon(t)|=\infty$ \cite{H69}.
Therefore for some $0<t_m<t_{\max}$,
\begin{align}
y^{\epsilon}(t_m)\in \partial B_{2C_z}(0) \label{con_3},
\end{align}
and
\begin{align}
y^{\epsilon}(t)\in B_{2C_z}(0), \quad \forall \ 0\le t<t_m. \label{con_4}
\end{align}
For $t< t_{\max}$,
 we have
\begin{align*}
&y^{\epsilon}(t)-x^z(t)=z_{\epsilon}-z+\int_0^t \zeta(\tau) d\tau,
\end{align*}
where $
\zeta(\tau) =f(\tau, y^{\epsilon}(\tau))+g(\epsilon,
\tau, y^{\epsilon}(\tau))-f(\tau, x^z(\tau))$.
It follows from (A3) that
\begin{align*}
|\zeta(\tau)| &= |\zeta(\tau)-g(\epsilon,
\tau, x^z(\tau))+ g(\epsilon,
\tau, x^z(\tau))|\\
& \le \mbox{Lip}(2C_z) |y^\epsilon(\tau) - x^z(\tau)|
+|g(\epsilon,\tau, x^z(\tau))|.
\end{align*}
Now for $0\le t<t_m$,
\begin{align*}
|y^{\epsilon}(t)-x^z(t)| \leq&\ |z_{\epsilon}-z|+\delta_{\epsilon}\\
 &+\int_0^t \mbox{Lip} (2C_z)|y^{\epsilon}(\tau)-x^z(\tau)|d\tau.
\end{align*}
Note that $\delta_{\epsilon}=\int_0^T|g\big(\epsilon, \tau, x^z(\tau)\big)|d\tau\rightarrow 0$ as $\epsilon\rightarrow 0$.
By Gronwall's lemma,
\begin{align*}
|y^{\epsilon}(t)-x^z(t)|\leq (\delta_{\epsilon}+|z_{\epsilon}-z|)e^{\mbox{Lip}(2C_z)t}
\end{align*}
for all $t\leq t_m$.
We can find $\bar{\epsilon}>0$ such that for all $\epsilon\leq \bar{\epsilon}$,
\begin{align*}
(\delta_{\epsilon}+|z_{\epsilon}-z|)e^{\mbox{Lip}(2C_z)T}<\frac{C_z}{3}.
\end{align*}
Then for all $0\leq t\leq t_m$,
$y^{\epsilon}(t)\in B_{3C_z/2}(0)$,
which is a contradiction to \eqref{con_3}.
We conclude for all $0<\epsilon \leq \bar{\epsilon}$, $y^{\epsilon}$ is defined on $[0, T]$. Next, \eqref{oe} follows readily.

ii) We have
\begin{align}
y^{\epsilon_i}(t)=z_{\epsilon_i}+\int_0^t \Big[f\big(\tau, y^{\epsilon_i}(\tau)\big)+g\big(\epsilon, \tau, y^{\epsilon_i}(\tau)\big)\Big]d\tau, \label{con_5}
\end{align}
and
\begin{align}
&|f\big(\tau, y^{\epsilon_i}(\tau)\big)
+g\big(\epsilon, \tau, y^{\epsilon_i}(\tau)\big)|\nonumber  \\
&\leq \mbox{Lip} (C_2)|y^{\epsilon_i}(\tau)|
 +|f(\tau, 0)+g(\epsilon, \tau, 0)| \nonumber \\
&\leq \mbox{Lip} (C_2)|y^{\epsilon_i}(\tau)|+C_1 \nonumber \\
& \leq \mbox{Lip}(C_2) C_2 +C_1 \label{con_6},
\end{align}
where $C_1$ is given in (A1).

By \eqref{con_5}--\eqref{con_6}, the functions $\{y^{\epsilon_i}(\cdot), i\geq 1\}$ are uniformly bounded and equicontinuous.
By Arzel\`{a}-Ascoli theorem \cite{Y80}, there exists a subsequence $\{y^{\epsilon_{i_j}}(\cdot), j\ge 1\}$ such that $y^{\epsilon_{i_j}}$ converges to $y^*\in C\big([0, T], \mathbb{R}^K\big)$ uniformly on $[0, T]$, as $j\to \infty$.
Hence,
\begin{align*}
y^*(t)=z+\int_0^t f\big(\tau, y^*(\tau)\big)+g\big(\epsilon, \tau, y^*(\tau)\big)d\tau
\end{align*}
for all $t\in [0, T]$. So \eqref{dot_x} has a solution.
\end{proof}

The  proof in part i) follows the method in
 \cite[sec. 2.4]{P96} and \cite[pp. 486]{S98}.

\section*{Appendix C}
\renewcommand{\theequation}{C.\arabic{equation}}
\setcounter{equation}{0}
\renewcommand{\thetheorem}{C.\arabic{theorem}}
\setcounter{theorem}{0}

{\it  Proof of Theorem \ref{theorem:iff}:}

Taking $\Pi_3=\Pi_4$ into account,
we rewrite the  system of \eqref{d1}, \eqref{d2} and
\eqref{d3}
 by use of a set of new variables
\begin{align*}
&\Lambda_1^N=\Pi_1(t), \ \Lambda_2^N=N\Pi_2(t),
\ \Lambda_3^N=N^2\Pi_3(t).\nonumber
\end{align*}
Here and hereafter $N$ is used as a superscript in various places.
This should be clear from the context.
We can determine functions $g_k$, $1\le k\le 3$, and obtain
\begin{align}
&\dot{\Lambda}_1^N = \Lambda_1^NM \Lambda_1^N-(\Lambda_1^NA+A^T\Lambda_1^N)
-Q \nonumber \\
 &\qquad
+ g_1(1/N, \Lambda_1^N, \Lambda_2^N),\label{Nd1_1} \\
& \Lambda^N_1(T)= (I-\frac{\Gamma_f^T}{N})
                 Q_f({I}-\frac{\Gamma_f}{N}),  \nonumber %
\end{align}
\begin{align}
&\dot{\Lambda}_2^N =  \Lambda_1^NM \Lambda_2^N+
\Lambda_2^N M \Lambda_1^N+
\Lambda_2^NM\Lambda_2^N \nonumber\\
&\qquad- (\Lambda_1^NG + \Lambda_2^N (A+G)
+A^T\Lambda_2^N) +Q\Gamma\nonumber \\
&\qquad + g_2(1/N,  \Lambda_2^N, \Lambda_3^N),  \label{Nd2_1} \\
&\Lambda^N_2(T)=-({I}-\frac{\Gamma_f^T}{N})Q_f\Gamma_f,  \nonumber
\end{align}
\begin{align}
&\dot{\Lambda}_3^N = (\Lambda_2^N)^TM \Lambda_2^N + \Lambda_3^NM \Lambda_1^N + \Lambda_1^NM \Lambda_3^N \nonumber \\
& \qquad +\Lambda_3^NM \Lambda_2^N+(\Lambda_2^N)^TM\Lambda_3^N\nonumber \\
                             &\qquad - \big((\Lambda_2^N)^T G+G^T\Lambda_2^N+\Lambda_3^N (A+G)+(A^T +G^T)\Lambda_3^N \big) \nonumber \\
                             &\qquad  - \Gamma^T Q\Gamma
                              +  g_3(1/N, \Lambda_2^N, \Lambda_3^N),  \label{Nd3_1}  \\
& \Lambda_3^N(T)=\Gamma_f^T Q_f\Gamma_f . \nonumber
\end{align}

In particular, we can determine
\begin{align*}
g_1= &\frac{1}{N} (1-\frac{1}{N})( \Lambda_2^N M \Lambda_2^N +(\Lambda_2^N)^T M (\Lambda_2^N)^T )\\
 & -\frac{1}{N} (\Lambda_1^N G + G^T\Lambda_1^N)
 - \frac{1}{N} (1-\frac{1}{N}) (\Lambda_2^N G+ G^T (\Lambda_2^N)^T)\\
 &+ \frac{1}{N}(\Gamma^T Q+Q\Gamma)-\frac{1}{N^2}\Gamma^T Q \Gamma.
\end{align*}
The expressions of $g_2$ and $ g_3$ can be determined in a similar way and the detail is omitted here.

Letting $N\rightarrow \infty$ in \eqref{Nd1_1}--\eqref{Nd3_1}, this gives a limiting  ODE system consisting of \eqref{d11}, \eqref{d21} and
\eqref{d3_1}.

If \eqref{d21} has a unique solution on $[0,T]$, we can uniquely solve $\Lambda_3$ from a linear ODE \eqref{d3_1}.
In view of $g_1, g_2, g_3$ and the terminal conditions in  \eqref{Nd1_1}--\eqref{Nd3_1},
by Theorem \ref{theorem:depen} i) and Remark \ref{remark:tc},
 there exists $N_0$ such that for all $N\ge N_0$, the system
 \eqref{Nd1_1}--\eqref{Nd3_1} has a solution on $[0,T]$
 and
 \begin{align}
 \sup_{N\ge N_0, 0\le t\le T} (|\Lambda_1^N| +|\Lambda_2^N|+ |\Lambda_3^N|)
  <\infty, \label{La123}
 \end{align}
which implies \eqref{main_con2} and so \eqref{main_conl1}. Consequently, asymptotic solvability follows.

Conversely, if asymptotic solvability holds,
there exists $N_0$ such that the system
 \eqref{Nd1_1}--\eqref{Nd3_1}
 has a solution on $[0,T]$
 for all $ N\ge N_0$ and  \eqref{La123} holds. By Theorem  \ref{theorem:depen} ii), \eqref{d21} has a unique solution on $[0,T]$.
This completes the proof of  Theorem \ref{theorem:iff}. \endproof

{\it Proof of Proposition \ref{prop:sr}:}

 We can check that
 $$(J_{23}^T\mbS_1,\ J_{23}^T\mbS_3,\ J_{23}^T\mbS_2,\ J_{23}^T\mbS_4, \cdots, J_{23}^T\mbS_N)$$
satisfies \eqref{DE3_S}.
Hence $\mbS_1= J_{23}^T \mbS_1$. We can further show $\mbS_1= J_{12}^T \mbS_2$. By the method in the proof of Lemma \ref{lem:Prep4}, we obtain the  representation for $\mbS_i$.
Next, for each $i$ we have
\begin{align*}
\begin{cases}
\dot{\mbr}_i(t)= \theta_1^T M \theta_1 +2(N-1) \theta_2^TM \theta_1
 -\mbox{Tr}( D^T \Pi_1 D)\\
 \qquad\quad -(N-1) \mbox{Tr}( D^T \Pi_3 D) -\eta^T Q \eta  , \\
  \mbr_i(T)= \eta_f^T Q_f \eta_f,
\end{cases}
\end{align*}
and therefore $\mbr_1=\cdots =\mbr_N$.
 \endproof

{\it Proof of Proposition \ref{prop:sichi}:}

 Recalling $M=BR^{-1}B^T$, by Proposition \ref{prop:sr} we derive
\begin{align}\label{theta_1}
\begin{cases}
\dot{\theta_1}(t) = \Pi_1M\theta_1+(N-1)(\Pi_2M
\theta_1+\Pi_2^TM\theta_2)\\
                            \qquad \quad-\big(A^T+\frac{G^T}{N}\big)
                            \theta_1-
                            \frac{N-1}{N}G^T\theta_2  \\
                           \qquad\quad+\big(I-\frac{
                           \Gamma^T}{N}\big)Q\eta,   \\
                            \theta_1(T)=-(I-\frac{
                           \Gamma_f^T}{N})Q_f\eta_f,
                           \end{cases}
                           \end{align}
                           and
                           \begin{align}\label{theta_2}
                           \begin{cases}
\dot{\theta_2}(t) = \big(\Pi_2^T+(N-1)\Pi_3\big)M\theta_1\\
                            \qquad\quad+\big(\Pi_1+(N-2)\Pi_2^T\big)M
                            \theta_2-\frac{1}{N}G^T\theta_1  \\
                            \qquad\quad-\big(A^T+\frac{N-1}{N}G^T\big)\theta_2
                            -\frac{1}{N}\Gamma^TQ\eta,   \\
                             \theta_2(T)=\frac{1}{N}\Gamma_f^T Q_f \eta_f. \end{cases}
\end{align}
Based on \eqref{theta_1}--\eqref{theta_2},  we may write
the ODEs of ${\chi}_1^N(t) $ and ${\chi}_2^N(t)$.
Under asymptotic solvability, we uniquely solve $(\Lambda_1, \Lambda_2, \Lambda_3, \chi_1, \chi_2)$ on $[0,T]$.
 We obtain \eqref{sichi} by writing the ODE system of
$(\Lambda_1^N, \Lambda_2^N, \Lambda_3^N, \chi_1^N, \chi_2^N)$ and next applying Theorem~\ref{theorem:depen}.
The proposition follows.~\endproof

\section*{Appendix D }
\renewcommand{\theequation}{D.\arabic{equation}}
\setcounter{equation}{0}
\renewcommand{\thetheorem}{D.\arabic{theorem}}
\setcounter{theorem}{0}

{\it  Proof of Proposition \ref{prop:global}:}

 i) If  $\hat Q\le 0$, \eqref{richa} is the Riccati ODE in a standard optimal control problem \cite{S98}, and so has a unique solution on $[0,T]$.

 ii)
The characteristic equation of \eqref{odeu} has solutions
$\hat\lambda_1=\hat a+ \alpha$,  $\hat\lambda_2=\hat a-\alpha$,
 where $\alpha=\sqrt{\hat \Delta}$.

   If  $\alpha>0$, we write
$u=C_1e^{\hat\lambda_1 t}+e^{\hat\lambda_2 t}$.
Then
$u'=C_1\hat\lambda_1 e^{\hat\lambda_1 t}+\hat\lambda_2 e^{\hat\lambda_2 t}$.
By $\Lambda_2(T)=0$, we obtain
$C_1=-({\hat\lambda_2}/{\hat\lambda_1}) e^{-2\alpha T}$ and
\begin{align*}
&\Lambda_2(t)=\frac{\hat\lambda_1\hat\lambda_2
\Big(e^{\alpha(T-t)}-e^{-\alpha(T-t)}\Big)}{\hat\lambda_2 e^{-\alpha(T-t)}-
\hat\lambda_1 e^{\alpha(T-t)}}\\
&\qquad\  =\frac{\hat Q \Big(e^{\alpha(T-t)}-e^{-\alpha(T-t)}\Big)}{\hat\lambda_2 e^{-\alpha(T-t)}-\hat\lambda_1 e^{\alpha(T-t)}},
\end{align*}
which exists on $[0,T]$.

 If $\alpha=0$,
 we write the solution of \eqref{odeu} as
$u=C_1 e^{\hat at}+te^{\hat at}$. This gives
$u'=C_1 \hat a e^{\hat at}+e^{\hat at}+t\hat ae^{\hat at}$.
Since $\Lambda_2(T)=0$, $C_1 \hat a+1+T \hat a=0$.       Therefore,
\begin{align*}
\Lambda_2(t)=-\frac{u'}{u} &= -\frac{C_1 \hat a+1+t\hat a}{C_1+t}
                           =\frac{\hat a^2(T-t)}{\hat a(t-T)-1},
\end{align*}
which exists on $[0,T]$.
\endproof

{\it  Proof of Proposition \ref{prop:max}:}

i) The computation is similar to the case in Proposition \ref{prop:global}
and we omit the details.

ii) The characteristic equation of \eqref{odeu} has solutions $\hat\lambda_{1,2}=\hat a\pm\beta i$. To solve \eqref{odeu},
we take
$u=C_1 e^{\hat at}\cos \beta t+e^{\hat at}\sin \beta t$.
 Now
\begin{align*}
u'=\hat ae^{\hat at}(C_1\cos\beta t+\sin\beta t)+\beta e^{\hat at}(-C_1\sin\beta t+\cos\beta t).
\end{align*}

Since $\Lambda_2(T)=0$, we determine
\begin{align*}
C_1 = & \frac{\hat a\sin\beta T+\beta\cos\beta T}{-\hat a\cos\beta T+\beta\sin\beta T} \\
       = &\frac{\cos\theta\sin\beta T+\sin\theta\cos\beta T}{-\cos\theta\cos\beta T+\sin\theta\sin\beta T} \\
=&        -\frac{\sin(\beta T+\theta)}{\cos(\beta T+\theta)}.
       \end{align*}
For this moment we suppose $\cos(\beta T+\theta)\ne 0 $ so that $C_1$ above is well defined.
Subsequently,
\begin{align*}
C_1\cos\beta t+\sin\beta t
=& \frac{\sin(\beta t-\beta T-\theta)}{\cos(\beta T+\theta)}, \\
-C_1\sin\beta t+\cos\beta t
=&\frac{\cos(\beta t-\beta T-\theta)}{\cos(\beta T+\theta)}.
\end{align*}
Therefore,
\begin{align}
\Lambda_2(t) =&- \frac{\hat a\sin(\beta t-\beta T-\theta)+\beta\cos(\beta t-\beta T-\theta)}{\sin(\beta t-\beta T-\theta)} \nonumber \\
      =& \frac{\sqrt{\hat{Q}}\sin\beta(t-T)}{\sin
      \big(\beta(T-t)+\theta\big)}. \label{lam2qsin}
\end{align}
If $\cos(\beta T+\theta)= 0$ occurs, we start by taking
$u= e^{\hat at}\cos \beta t+C_2e^{\hat at}\sin \beta t$.
We may determine $C_2=0$ and still obtain the same form of $\Lambda_2$
as in \eqref{lam2qsin}. \endproof

\section{Acknowledgment}
We would like to thank  the anonymous reviewers and the Associate Editor for
very helpful suggestions.

\bibliographystyle{plain}
\bibliography{TAC18-565bib}

\begin{IEEEbiography}{Minyi Huang} (S'01-M'04)
received the B.Sc. degree from Shandong University, Jinan, Shandong,
China, in 1995, the M.Sc. degree from the Institute
of Systems Science, Chinese Academy of Sciences,
Beijing, in 1998, and the Ph.D. degree from the
Department of Electrical and Computer Engineering,
McGill University, Montreal, QC, Canada, in 2003,
all in systems and control.

He was a Research Fellow first in the Department
of Electrical and Electronic Engineering, the University
of Melbourne, Melbourne, Australia, from February
2004 to March 2006, and then in the Department of Information Engineering,
Research School of Information Sciences and Engineering, the Australian
National University, Canberra, from April 2006 to June 2007. He joined
the School of Mathematics and Statistics, Carleton University, Ottawa, ON,
Canada as an Assistant Professor in July 2007, where he is now a
Professor. His research interests include mean field stochastic
control and dynamic games, multi-agent control and computation in distributed
networks with applications.
\end{IEEEbiography}

\begin{IEEEbiography}{Mengjie Zhou}
received the B.Sc. degree in applied mathematics from  Shanghai Jiao Tong University, Shanghai, China, in 2014, the M.Sc. degree in applied mathematics from  Western University, London, ON, Canada, in 2015, and the M.Sc. degree in probability and statistics from  Carleton University, Ottawa, ON, Canada, in 2017. She is currently pursuing the Ph.D. degree in probability and statistics under the supervision of Dr. Minyi Huang at  Carleton University.

   Her research interests include mean field stochastic control and Markov decision processes.
\end{IEEEbiography}

\end{document}